\documentclass[10pt]{amsart}

\usepackage{amsmath}
\usepackage{mathtools}
\usepackage{dsfont}
\usepackage{mathabx}

\usepackage[usenames, dvipsnames]{color}
\usepackage[colorlinks=true,
        raiselinks=true,
        linkcolor=MidnightBlue,
        citecolor=Brown,
        urlcolor=ForestGreen,
        pdfauthor={},
        pdftitle={},
        pdfkeywords={},
        pdfsubject={},
        plainpages=false]{hyperref}
\usepackage[T1]{fontenc}

\usepackage[norelsize, ruled, vlined]{algorithm2e}

\usepackage{enumitem}
\setlist{noitemsep, topsep=0cm}

\usepackage{tikz}
\usepackage{float}
\usetikzlibrary{fadings}
\usetikzlibrary{shapes.misc}
\tikzset{cross/.style={cross out, draw=black, minimum size=2*(#1-\pgflinewidth), inner sep=0pt, outer sep=0pt},cross/.default={1pt}}
\usetikzlibrary{shapes,arrows}
\usetikzlibrary{positioning}
\usetikzlibrary{plotmarks}

\allowdisplaybreaks

\newtheorem{theorem}{Theorem}[section]
\newtheorem{proposition}[theorem]{Proposition}
\newtheorem{lemma}[theorem]{Lemma}
\newtheorem{corollary}[theorem]{Corollary}
\newtheorem{definition}[theorem]{Definition}

\theoremstyle{remark}
\newtheorem{remark}[theorem]{Remark}
\theoremstyle{claim}
\newtheorem{claim}[theorem]{Claim}

\theoremstyle{definition}

\newcommand{\norm}[1]{\left\lVert#1\right\rVert}
\newcommand{\pmat}[3]{\begin{pmatrix} #1 & #2 &\cdots & #3 \end{pmatrix}}
\newcommand{\pmatr}[1]{\begin{pmatrix}#1\end{pmatrix}}
\newcommand{\dict}[1]{d_{#1}}
\newcommand{\dictx}[1]{\delta_{#1}}
\newcommand{\inprod}[2]{\left\langle#1, #2\right\rangle}
\newcommand{\symmetric}[1]{\mathbb{S}^{#1 \times #1}}
\newcommand{\posdef}[1]{\mathbb{S}^{#1 \times #1}_{++}}
\newcommand{\possemdef}[1]{\mathbb{S}^{#1 \times #1}_{+}}
\newcommand{\ortho}[1]{\mathbb{O}^{#1 \times #1}}
\newcommand{\linearmap}{A}
\newcommand{\orthobasis}[1]{x_{#1}}
\newcommand{\rodvec}[1]{y_{#1}}
\newcommand{\length}[1]{\alpha_{#1}}
\newcommand{\majorize}{\prec}
\newcommand{\dimension}{n}
\newcommand{\permpolytope}[1]{\mathcal{P}(#1)}
\newcommand{\indexes}[1]{1, 2, \ldots, #1}
\newcommand{\eigval}[1]{\lambda_{#1}}
\newcommand{\eigvec}[1]{u_{#1}}
\newcommand{\primalvar}[1]{\lambda_{#1}}
\newcommand{\dualvar}[1]{x_{#1}}
\newcommand{\lagrangian}{L}
\newcommand{\poseig}{m}
\newcommand{\sigmat}[1]{\Sigma_{#1}}
\newcommand{\eigvalsig}[1]{\sigma_{#1}}
\newcommand{\eigvecsig}[1]{u_{#1}}

\newcommand{\xz}{\bar{x}}
\newcommand{\xf}{\hat x}
\newcommand{\reachindex}{K}
\newcommand{\reachmat}{\mathcal{R}}

\newcommand{\uvec}[1]{u_{#1}}
\newcommand{\bval}[1]{b_{#1}}
\newcommand{\aval}[1]{a_{#1}}
\newcommand{\gmap}[1]{g(#1)}

\newcommand{\optval}[3]{J \left( #1, #2, #3 \right)}
\newcommand{\eigvalmap}[3]{\lambda \left( #1, #2, #3 \right)}
\newcommand{\eqmul}{\eta}
\newcommand{\ineqmul}[1]{\gamma_{#1}}
\newcommand{\partition}[1]{n_{#1}}
\newcommand{\randomvec}{V}
\newcommand{\const}{c}
\newcommand{\projectedrv}{\randomvec (\const)}
\newcommand{\coef}[1]{r_{#1}}
\newcommand{\coefvec}{r}
\newcommand{\scheme}{f}
\newcommand{\optscheme}{f \opt}

\DeclareSymbolFont{symbolsC}{U}{pxsyc}{m}{n}
\DeclareMathOperator{\trace}{tr}
\DeclareMathOperator{\image}{image}
\DeclareMathOperator{\Span}{span}
\DeclareMathOperator{\rank}{rank}

\DeclareMathOperator{\diag}{diag}
\DeclareMathOperator{\sort}{sort}
\DeclareMathOperator{\EE}{\mathsf{E}}
\DeclareMathOperator{\PP}{\mathsf{P}}
\DeclareMathOperator{\var}{Var}
\DeclareMathOperator*{\minimize}{minimize}
\DeclareMathOperator*{\sbjto}{subject\,to}
\DeclareMathOperator*{\maximize}{maximize}

\DeclarePairedDelimiterX\set[1]\lbrace\rbrace{\def\suchthat{\;\delimsize\vert\;}#1}

\newcommand{\Let}{\coloneqq}

\newcommand{\R}{\mathbb{R}}
\newcommand{\opt}{^\ast}
\newcommand{\transp}{^\top}
\newcommand{\lra}{\longrightarrow}
\newcommand{\lmt}{\longmapsto}
\newcommand{\dsize}{K}
\newcommand{\dictionary}{D_{\length{}}}

\renewcommand{\geq}{\geqslant}
\renewcommand{\ge}{\geqslant}

\renewcommand{\leq}{\leqslant}
\renewcommand{\mapsto}{\longmapsto}

\title[A complete characterization of \(\ell_2\)-optimal dictionaries]{A Complete Characterization of Optimal Dictionaries for Least Squares Representation}
\author{Mohammed Rayyan Sheriff}
\author{Debasish Chatterjee}
\address{Systems and Control Engineering\\ IIT Bombay, Powai\\ Mumbai 400076\\ India.}
\email{(MRS) \texttt{rayyan@iitb.ac.in}, (DC) \texttt{dchatter@iitb.ac.in}.}

\begin{document}

\begin{abstract}%   <- trailing '%' for backward compatibility of .sty file
	Dictionaries are collections of vectors used for representations of elements in Euclidean spaces. While recent research on optimal dictionaries is focussed on providing sparse (i.e., \(\ell_0\)-optimal,) representations, here we consider the problem of finding optimal dictionaries such that representations of samples of a random vector are optimal in an \(\ell_2\)-sense. For us, optimality of representation is equivalent to minimization of the average $\ell_2$-norm of the coefficients used to represent the random vector, with the lengths of the dictionary vectors being specified a priori. With the help of recent results on rank-\(1\) decompositions of symmetric positive semidefinite matrices and the theory of majorization, we provide a complete characterization of \(\ell_2\)-optimal dictionaries. Our results are accompanied by polynomial time algorithms that construct \(\ell_2\)-optimal dictionaries from given data.
\end{abstract}


\maketitle

%%%%%%%%%%%%%%%%%%%%%%%%%%%%%%%%%%%%%%%%%%%%%%%%%%%%%%%%%%%%%%%%%%%%%%%%%%%%%%%%
\section{Introduction}
	
We begin with a toy example to motivate the problems treated in this article. Let \( V \) be a random vector that attains values `close' to \(\pmatr{0 & 2}^\top\) with high probability; Suppose that our dictionary consists of the vectors $ d_1 = \pmatr{1 & -\epsilon}^\top $ and $ d_2 = \pmatr{1  &\epsilon}^\top $ in $\mathbb{R}^2$, with a small positive value of $\epsilon$. Since we must represent \( V \) using $d_1$ and $d_2$, the corresponding coefficients \( \alpha_1 \) and \( \alpha_2 \) must be such that $\alpha_1 \pmatr{1 & \epsilon}^\top + \alpha_2 \pmatr{1 & -\epsilon}^\top = V \approx \pmatr{0 & 2}^\top $. A quick calculation shows that the magnitudes of the coefficients \( \alpha_1 \) and \( \alpha_2 \) should then be approximately equal to \( 1/(\epsilon) \) with high probability. To wit, the magnitudes of these coefficients are large for small values of $\epsilon$. It is therefore more appropriate in this situation to consider a dictionary consisting of vectors $ d_1\opt = \pmatr{\epsilon & 1}^\top $ and $ d_2\opt = \pmatr{-\epsilon & 1}^\top $ to represent the samples of \( V \), in which case, the magnitudes of the coefficients of the representations are closer to $1$ with high probability. The latter values are far smaller compared to the values close to $1/(\epsilon)$ obtained with the preceding dictionary. This simple example shows that given some statistical information about the random vectors to be represented, the question of designing a dictionary that minimizes the average cost of representation can be better addressed. 

Our motivation for the investigation carried out in this article and \cite{sheriff2016optimal} comes from a control theoretic perspective. Consider a linear time-invariant control system modeled by the recursion
\begin{equation}
\label{linear system}
x(t+1) = A x(t) + B u(t), \quad t = 0, 1, \ldots,
\end{equation}
where the `system matrix' \(A\in\R^{n\times n}\) and the `control matrix' \(B\in\R^{n\times m}\) are given, with the initial boundary condition \(x(0) = \xz\in\R^n\) fixed. For an arbitrarily selected \(\xf\in\R^n\), the standard / classical \emph{reachability problem} for \eqref{linear system}, consists of finding a sequence \( (u(t))_{t}\subset\R^m \) of control vectors that steer the system states to \( \xf \). A necessary and sufficient condition for such a sequence to exist for every pair \((\xz, \xf)\) is that there exists a positive integer \(K\) such that the rank of the matrix \( \reachmat_{\reachindex} (A, B) \Let \pmatr{B & AB & \cdots & A^{K-1}B} \) is equal to \( n \). We impose this rank condition for the moment, and pick an integer \(K\ge n\). It is observed at once that the control vectors \( (u(t))_{t=0}^{\reachindex-1} \) needed to execute the transfer of the states of \eqref{linear system} from $\xz$ to $\xf$ must solve the linear equation
\[
	\xf - A^\reachindex \xz = \sum_{t=0}^{\reachindex-1} A^t B u(t) = \reachmat_{\reachindex} (A, B) \pmatr{u(\reachindex-1)\\\vdots \\ u(1)\\ u(0)}.
\]
Out of all the feasible control sequences \( (u(t))_{t} \) that execute the system state transfer, it is now natural to consider those particular control sequences that minimizes the `control cost' $ \sum_{i = 0}^{\reachindex - 1} \norm{u(t)}^2 $ of transferring \(\xz\) to \(\xf\). In practice,  the afore mentioned \( \ell_2 \) performance index provides a natural measure of the energy spent to 
steer the system from \(\bar x\) to \(\hat x\). 

Minimization of control effort has been an integral part of control theory, and its practical importance can hardly be overstated in this context. This topic is generally studied under the class of Linear Quadratic optimal control problems; see, e.g., \cite{bertsekas1995dynamic}, \cite{anderson2007optimal}, \cite{clarke2013functional}, \cite{liberzon2012calculus}, or any standard book on optimal control. Our initial interest in this setting was to find optimal orientations of the columns of \( \reachmat_{\reachindex} (A, B) \), thus as a means of characterizing systems that are `better' in a structural sense than others, where the criterion for optimality is to optimize a certain measure of quality / figure of merit of a linear system. For instance, if we define $W_{A, B} \coloneqq \reachmat_{\reachindex}(A, B)\reachmat_{\reachindex}(A, B)\transp$, the early work \cite{ref:MulWeb-72} proposed the quantities
$\trace\bigl(W^{-1}_{A, B}\bigr)$, \( \lambda_{\min}^{-1} \big( W_{A, B} \big) \) and \( \det \big( W_{A, B} \big) \) as three measures of quality, and one would like to achieve a minimal value of $\trace\bigl(W^{-1}_{A, B}\bigr)$, \( \lambda_{\min}^{-1} \big( W_{A, B} \big) \) and a maximal value of \( \det \big( W_{A, B} \big) \) for ``good'' control systems based on energy considerations. Recently, we demonstrated in \cite{sheriff2017frame} that all the aforementioned three measures of quality get optimized simultaneously when the orientation of columns of \( \reachmat_{\reachindex} (A, B) \) is tight, i.e., when the columns of \( \reachmat_{\reachindex} (A, B) \) form a \emph{tight frame}. A succinct connection between good structural properties of a linear system and frame theory was, therefore, established.

It is of independent interest to address the mathematical problem that lies at the heart of the above discussion. To this end, we define a \emph{dictionary} to be a collection of vectors in a finite-dimensional vector space over $\mathbb{R}$, with which every element of the vector space can be represented. A dictionary is a generalization of a basis: While the number of vectors in a basis is exactly equal to the dimension of the vector space, a dictionary may contain more elements. Analogous to the discussion on linear systems, our objective is to find a dictionary that offers optimal least squares representation, and we shall refer such a dictionary as an \( \ell_2 \)-optimal dictionary. Characterization and algorithms to compute \( \ell_2 \)-optimal dictionaries of unit length vectors, optimal in representation of a class of vectors / random vector distributed according a generic distribution \( \mathsf{P} \) were provided in \cite{sheriff2016optimal}. In particular, it was found that the unit norm tight frames are \( \ell_2 \)-optimal for the representation of samples distributed uniformly over the unit sphere.

In the relatively recent article \cite{casazza2006physical} we encountered the problem of finding conditions for the existence of tight frames with arbitrary length vectors  and their characterization. In this article our objective is similar to that of \cite{casazza2006physical}, we consider the \( \ell_2 \)-optimal dictionary problem in a general setting where the dictionary vectors are constrained to be of fixed lengths that can be any arbitrary positive numbers instead of unity. It turns out that the results of \cite{sheriff2016optimal} can neither be directly used nor tweaked to find an \( \ell_2 \)-optimal dictionary of variable length vectors, and a fresh investigation is needed. Our approach in the current article centers around the theory of \emph{majorization}, which was fruitfully employed earlier in \cite{casazza2006physical} and \cite{casazza2002frames} in the context of frames.

In this article we start in a general setting of solving the problem of \(\ell_2\)-optimal representations of random vectors in \(\R^d\) with distribution \(\mathsf P\). For the problem to be well-defined, we need the distribution \(\PP\) to have finite variance, which we assume. In this setting:
\begin{itemize}[leftmargin = *]
\item We provide an almost explicit solution to the \( \ell_2 \)-optimal dictionaries in terms of a rank-\( 1 \) decomposition of a certain positive matrix.

\item An algorithm to compute the \( \ell_2 \)-optimal dictionaries in polynomial time is also provided.

\item To compute \(\ell_2\)-optimal dictionaries, it turns out that only the mean and the variance of the distribution \( \PP \) have to be learned / known to arrive at a complete solution. This is an advantage in situations where complete precise information about the underlying distribution may not be available.

\item Finally, we demonstrate that the \(\ell_2\)-optimal dictionaries are robust with respect to the errors in the estimation / learning of the values of the mean and the variance, which is a desirable property.
\end{itemize}

This article unveils as follows: In Section \ref{The Dictionary Learning problem} we formally introduce our problem of finding an optimal dictionary which offers least squares representation. Section \ref{The Dictionary Learning problem} is the heart of this article, where we solve the problem of finding an \(\ell_2\)-optimal dictionary, and arrive at an almost explicit solution. Algorithms to construct \(\ell_2\)-optimal dictionaries are given in Section \ref{The Dictionary Learning problem} after Theorem \ref{theorem:DL-theorem-2}. The case of representing random vectors distributed uniformly on the unit sphere is treated in Subsection \ref{uniform section}; we demonstrate that the \(\ell_2\)-optimal dictionaries in this case are \emph{finite tight frames}. In the intermediate Section \ref{section:Mathematical tools and other auxiliary results}, we recollect some of the standard results in the theory of majorization and also provide some auxiliary results essential for the solutions of our main results. In the later sections \ref{section:proofs-of-main-results} and \ref{section:proof-of-auxiliary-results}, we provide proofs of the main and auxiliary results, respectively.

%%%%%%%%%%%%%%%%%%%%%%%%%%%%%%%%%%%%%%%%%%%%%%%%%%%%%%%%%%%%%%%%%%%%%%%%%%%%%%%%%%%%%%%%%%%%%%%%%
\subsection*{Notations}
We employ standard notations in this article. The Euclidean norm is denoted by \(\norm{\cdot}\). the \(n\times n\) identity and \(m\times n\) zero matrices are denoted by \(I_n\) and \(O_{m\times n}\), respectively. For a matrix \(M\) we let \(\trace(M)\) and \(M^+\) denote its trace and Moore-Penrose pseudo-inverse, respectively. The set of \(\dimension \times \dimension \) symmetric matrices with real entries is denoted by \( \symmetric{\dimension} \), and the set of \(\dimension \times \dimension \) symmetric and positive (semi-)definite matrices with real entries is denoted by \( \posdef{\dimension} \) (\( \possemdef{\dimension} \)). For a Borel probability measure \({\mathsf{P}}\) defined on \(\R^n\), we let \(\EE_{\mathsf{P}}[\cdot]\) denote the corresponding mathematical expectation. The image of a map \(f\) is written as \(\image(f)\). The gradient of a continuously differentiable function \(f\) is denoted by \(\nabla f\). The sequence \( (\delta_i)_{i = 1}^n \) denotes the standard Euclidean basis of \( \R^{\dimension} \).

Let \( \poseig \), \( T \) be positive integers such that \( T \leq \poseig \), and let \( (c_i)_{i = 1}^{\poseig} \) and \( (a_i)_{i = 1}^{\poseig} \) be two sequences of positive real numbers. Let \( (n_l)_{l= 1}^T \subset \{ \indexes{\poseig} \} \) be such that \( 1 \eqqcolon n_1 \leq n_2 \leq \cdots \leq n_T \leq \poseig \). Let us define the following two maps, the first being
\begin{subequations}
\begin{align}
\label{eq:definition-of-lambda-a}
(\eigval{i})_{i = 1}^{\poseig} = \eigvalmap{(c_i)_{i = 1}^{\poseig}}{(a_i)_{i = 1}^{\poseig}}{(n_l)_{l= 1}^T}, \quad \text{where}
\end{align}
\begin{align}
\label{eq:definition-of-lambda}
\eigval{i} \Let c_i \left( \frac{\sum\limits_{j = n_l}^{n_{l + 1} - 1} a_j}{\sum\limits_{j = n_l}^{n_{l + 1} - 1} c_j} \right);
\end{align}
\end{subequations}
for \( n_l \leq i < n_{l + 1} \), \( l = \indexes{T} \) with \( n_{T + 1} = \poseig + 1 \), and the second being
\begin{equation}
\label{eq:J-map-definition}
\optval{(c_i)_{i = 1}^{\poseig}}{(a_i)_{i = 1}^{\poseig}}{(\partition{l})_{l = 1}^{T}} \Let \sum\limits_{l = 1}^T \left( \frac{ \left( \sum\limits_{j = n_l}^{n_{l + 1} - 1} c_j \right)^2 }{\sum\limits_{j = n_l}^{n_{l + 1} - 1} a_j} \right).
\end{equation}
These two maps will be employed many times in the sequel, in particular in Theorems \ref{theorem:DL-theorem-1} and \ref{theorem:DL-theorem-2}.

Let \( k,\dsize \) be any two positive integers, and \( A : \R^{\dsize} \longrightarrow \R^{\dsize} \) be any linear map. Let \( (\uvec{i})_{i = 1}^{k} \) be any arbitrary sequence of vectors in \( \R^{\dsize} \), then the sequence \( (v_i)_{i = 1}^{k} \Let A \sort \{ (\uvec{i})_{i = 1}^{k} \} \) is defined to be the permutation of \( (\uvec{i})_{i = 1}^{k} \) such that 
\begin{equation}
\label{eq:definition-of-Asort}
\inprod{v_1}{A v_1} \geq \inprod{v_2}{A v_2} \geq \cdots \geq \inprod{v_k}{A v_k}. 
\end{equation}

%%%%%%%%%%%%%%%%%%%%%%%%%%%%%%%%%%%%%%%%%%%%%%%%%%%%%%%%%%%%%%%%%%%%%%%%%%%%%%%%%%%%%%%%%%%%%%%%%%%

\section{The $\ell_2$-optimal dictionary problem and its solution}
\label{The Dictionary Learning problem}
Let $V$ denote an \( \R^{\dimension} \)-valued random vector defined on some probability space, and having distribution (i.e., Borel probability measure,) $\PP$. We assume that the variance of \(V\) is well defined. Let $R_{\randomvec}$ denote the support of $\PP$,\footnote{Recall \cite[Theorem 2.1, Definition 2.1, pp.\ 27-28]{ref:Par-05} that the support of \( \PP \) is the set of points \(z\in\R^n\) such that the \( \PP \)-measure of every open neighbourhood of \(z\) is positive.}  and let \( \const \in \R^{\dimension} \) be a \emph{constant} vector. Let us define the following quantities
\begin{equation}
\label{eq: rv xv definition}
\begin{aligned}
R_{\randomvec}(\const) & \Let \left\{ v \in \R^{\dimension} \vert (v + \const) \in R_{\randomvec}  \right\} \\ 
X_{\randomvec}(\const) & \Let \text{the smallest subspace of \( \R^{\dimension} \)  containing  \( R_{\randomvec}(\const) \).}
\end{aligned}
\end{equation}

Let \( \dsize \geq \dimension \)  be a positive integer and \( \length{} \Let (\length{i})_{i = 1}^{\dsize} \) be a non-increasing sequence  of positive real numbers. Our goal is to represent the instances/samples of $V$ with the help of a \emph{dictionary} of vectors:
\[
	\dictionary \Let \set[\big]{\dict{i} \in \R^n \suchthat \inprod{\dict{i}}{\dict{i}} = \length{i} \text{ for} \; i = 1,\ldots, \dsize},
\]
in an optimal fashion, the criteria for optimality will be defined momentarily. Every instance \( v \) of random vector \( \randomvec \) is represented by the variation \( (v - \const) \) of \( \randomvec \) from the constant \( \const \) for obvious advantages.  A \emph{representation} of an instance \( v \) of the random vector \( V \) is given by the coefficient vector  \( \coefvec \Let ( \coef{1} \; \cdots \; \coef{\dsize} )^{\transp}  \), such that
\begin{equation}
\label{21}
(v - \const) = \sum_{i = 1}^\dsize \coef{i} \dict{i}.
\end{equation}
A \emph{reconstruction} of the sample \( v \) from the representation \( \coefvec \) is carried out by taking the linear combination \( \const \; + \; \sum_{i = 1}^{\dsize} \coef{i} \dict{i} \). We define the \emph{cost} associated with representing \( v \) in terms of the coefficient vector \( \coefvec \) as \( \sum_{i = 1}^{\dsize} \coef{i}^2 \). Since the dictionary vectors \( (\dict{i} )_{i = 1}^{\dsize} \) must be able to represent any sample of \( V \), the property that \( \Span ( \dict{i} )_{i = 1}^{\dsize} \supset R_{\randomvec}(\const) \) is essential. A dictionary \( \dictionary = ( \dict{i} )_{i = 1}^{\dsize} \subset\R^{\dimension} \) is said to be \emph{feasible} if \( \Span ( \dict{i} )_{i = 1}^{\dsize} \supset R_{\randomvec} (\const)$. We denote by \( \mathcal{\dictionary}(\const) \) the set of all feasible dictionaries to represent \( \randomvec \) with a constant of representation \( \const \).

For any vector \( v \in R_{\randomvec} (\const) \) and a feasible dictionary \( \dictionary \) of vectors \( ( \dict{i} )_{i=1}^{\dsize} \) such that if \( \poseig \Let \dim \left( \Span(\dict{i})_{i = 1}^{\dsize} \right) \), then the linear equation \eqref{21} is satisfied by infinitely many values of \( \coefvec \) whenever \( \dsize  > \poseig \). In fact, the solution set of \eqref{21} constitutes a \( (\dsize - \poseig ) \)-dimensional affine subspace of \( \R ^{\dsize} \). Therefore, in order to represent a given $v$ uniquely, one must define a mechanism of selecting a particular point from this affine subspace, thus making the coefficient vector \( \coefvec = (\coef{1} \; \cdots \; \coef{\dsize})^{\transp} \) a function of $v$. Let \( \scheme \) denote such a function; to wit, \( \scheme(v) \Let \coefvec \) is the coefficient vector used to represent the sample $v$. We call such a map \( R_{\randomvec} \ni v  \longmapsto \scheme(v) \in \R^{\dsize} \) a \emph{scheme of representation}. For a constant \( \const \in \R^{\dimension} \), representation of samples of the random vector \( \randomvec \) using a dictionary \( \dictionary \in \mathcal{\dictionary} (\const) \) and a scheme \( \scheme \) is said to be \emph{proper} if every vector \( v \in R_{\randomvec} (\const) \) can be uniquely represented and then exactly reconstructed back. It is clear that for proper representation of $\randomvec$ with a dictionary \( \dictionary \) consisting of vectors \( ( \dict{i} )_{i = 1}^{\dsize} \) and the constant \( \const \), the mapping  \( R_{\randomvec} \ni v  \longmapsto \scheme(v) \in \R^{\dsize} \) should be an injection that satisfies
\begin{equation}
\label{injectivity}
(\randomvec - \const) = \pmat{ \dict{1} }{ \dict{2} }{ \dict{\dsize} } \scheme (\randomvec) \quad \text{$\PP$-almost surely.}
\end{equation}
A scheme of representation \( \scheme \) is said to be \emph{feasible} if \eqref{injectivity} holds. We denote by \( \mathcal{F} (\const, \dictionary) \) the set of all feasible schemes for representing \( \randomvec \) using a constant \( \const \) and a dictionary \( \dictionary \). 

Given a scheme \( \scheme \) of representation, the (random) cost associated with representing \( \randomvec \) is given by \( \inprod{\scheme (\randomvec)}{\scheme (\randomvec)}  \). The problem of finding an \(\ell_2\)-optimal dictionary can now be posed as:
\begin{quote}
	Find a triplet consisting of a constant vector \( \const \opt \in \R^{\dimension} \), a dictionary \( \dictionary \opt \in \mathcal{\dictionary} (\const \opt) \) and a scheme \( \optscheme \in \mathcal{F} (\const \opt, \dictionary \opt) \) of representation such that the average cost \( \EE_{\PP} \bigl[ \inprod{\optscheme (\randomvec)}{\optscheme (\randomvec)} \bigr] \) of representation is minimal.
\end{quote}
Here the subscript $\PP$ indicates the distribution of the random vector \( \randomvec \) with respect to which the expectation is evaluated. In other words, we have the following optimization problem:
\begin{equation}
	\label{eq:DL-problem-general}
	\begin{aligned}
		& \minimize_{\const, \dictionary, \scheme}	&& \EE_{\PP} \bigl[ \inprod{\scheme (\randomvec)}{\scheme (\randomvec)} \bigr]\\
		& \sbjto				&& 
		\begin{cases}
		    \const \in \R^{\dimension},  \\
			\dictionary \in \mathcal{\dictionary} (\const), \\
			\scheme \in \mathcal{F} (\const, \dictionary).			
		\end{cases}
	\end{aligned}
\end{equation}
The problem \eqref{eq:DL-problem-general} is the \emph{\( \ell_2 \)-optimal dictionary} problem. Due to the constraints on the dictionary vectors and the restriction on the feasible schemes, it is obvious that the \( \ell_2 \)-optimal dictionary problem \eqref{eq:DL-problem-general} is non-convex.

%due to the constraint that the dictionary vectors $ \{ \dict{i} \}_{i = 1}^K $ of a feasible dictionary must be of unit length. Even if we change this constraint to $ \left\{ \norm{ \dict{i} } \leq 1 \right\}$ from $ \left\{ \norm{ \dict{i} } = 1 \right\} $, which makes the feasible region of dictionary vectors convex, the set of feasible schemes of representation is not known to be a convex set a priori.

In this article we solve the $\ell_2$-optimal dictionary problem given \eqref{eq:DL-problem-general} in two steps:
\begin{enumerate}[label=(Step \Roman*), leftmargin=*, align=left, widest=II]
	\item We first assume that \( \const = 0 \) and \( X_{\randomvec}(\const) = X_{\randomvec}(0) = \R^{\dimension} \).
	\item We let \( \const \) be any vector in \( \R^{\dimension}\) and \( X_{\randomvec}(\const) \) be any proper nontrivial subspace of \(\R^{\dimension}\).\footnote{The trivial case of \(X_{\randomvec}(\const) = \{0\}\) is discarded because then there is nothing to prove; and therefore we limit ourselves to `nontrivial' subspaces of \(\R^{\dimension}\).} 
\end{enumerate}
The remainder of this section is devoted to describing Steps I and II by exposing our main results, followed by discussions, a numerical example, and a treatment of the important case of the uniform distribution on the unit sphere of \( \R^{\dimension} \).

%===============================================================================
\subsection{Step I: \( \const = 0 \) and \( X_{\randomvec} (\const) = \R^{\dimension}\)}

If \( X_{\randomvec} (0) = \R^{\dimension} \), a dictionary of vectors \( \dictionary = ( \dict{i} )_{i = 1}^{\dsize} \subset \R^{\dimension} \) is feasible if and only if \( \inprod{\dict{i}}{\dict{i}}  = \length{i} \) for all \( i = 1,\ldots, \dsize \), and \( \Span ( \dict{i} )_{i = 1}^{\dsize} = \R^{\dimension} \). Thus, \eqref{eq:DL-problem-general} with \( \const = 0 \) reduces to:
\begin{equation} 
	\label{eq:DL-problem}
	\begin{aligned}
		& \minimize_{\{d_i\}_{i=1}^{\dsize}, \scheme}	&& \EE_{\PP} \bigl[ \inprod{\scheme (\randomvec)}{\scheme (\randomvec)} \bigr]\\
		& \sbjto							&&
		\begin{cases}
			\inprod{\dict{i}}{\dict{i}} = \length{i} \text{ for all \( i = 1, \ldots, \dsize, \)}\\
			\Span ( \dict{i} )_{i = 1}^{\dsize} = \R^{\dimension}, \\
			\pmatr{\dict{1} & \dict{2} & \cdots & \dict{\dsize}} \scheme (\randomvec) = \randomvec \text{\ \ \(\mu\)-almost surely}.
		\end{cases}
	\end{aligned}
\end{equation}

Let \( \sigmat{\randomvec} \Let \EE_{\PP} [\randomvec \randomvec\transp] \). We observe that \( \sigmat{\randomvec} \) is positive definite: Indeed, if not, then there exists a nonzero vector \( x \in \R^{\dimension} \) such that \( x\transp \randomvec = 0 \) almost surely, which contradicts the assumption that \( X_{\randomvec}(0) = \R^{\dimension} \). 

Existence and characterization of the optimal solutions to \eqref{eq:DL-problem} is asserted by the following:

\begin{theorem}
\label{theorem:DL-theorem-1}
Let \( \sigmat{\randomvec} \Let \EE_{\mathsf{P}}[\randomvec \randomvec\transp] \), \( (\eigvalsig{i})_{i = 1}^{\dimension} \) be the eigenvalues of \( \sigmat{\randomvec} \) arranged in non-increasing order. Let the sequence of positive real numbers \( (\length{i}')_{i = 1}^{\dimension} \) be defined by \( ( \length{i}' )_{i = 1}^{\dimension - 1} \Let ( \length{i} )_{i = 1}^{\dimension - 1} \) and \( \length{\dimension}' \Let \sum_{j = \dimension}^{\dsize} \length{j} \). 
The optimization problem:
\begin{equation}
\label{eq:dual-problem-in-DL-theorem-1}
\begin{aligned}
		& \minimize_{( \dualvar{t} )_t \subset \R}	&&  \sum_{t = 1}^{\dimension}  \length{t}' \dualvar{t}^2 - 2  \sqrt{\eigvalsig{t}} \dualvar{t}  \\
		& \sbjto						&&
		0 \leq \dualvar{1} \leq \cdots \leq \dualvar{\dimension},
\end{aligned}
\end{equation}
admits a unique optimal solution \( (\dualvar{t} \opt)_{t = 1}^{\dimension} \). Let the optimal value of \eqref{eq:dual-problem-in-DL-theorem-1} be \( q \opt \), define an ordered set \( (\partition{1}, \partition{2}, \ldots, \partition{T}) \subset (\indexes{\dimension}) \) iteratively by
\[%begin{equation}
        \begin{aligned}
        %\label{eq: partition}
        \partition{1} & \Let 1, \\
        \partition{l} & \Let \min \{ t \; \vert \; \partition{(l - 1)} < t \leq \dimension, \; \dualvar{(t - 1)} \opt < \dualvar{t} \opt \} \quad  \text{for all \( l = 2,\ldots, T \),} 
        \end{aligned}
\]%end{equation}
and let \( (\eigval{i} \opt)_{i = 1}^{\dimension} \) be the non-increasing sequence of positive real numbers defined by
\[
(\eigval{i} \opt)_{i = 1}^{\dimension} \Let \eigvalmap{(\sqrt{\eigvalsig{i}})_{i = 1}^{\dimension}}{(\length{i}')_{i = 1}^{\dimension}}{(\partition{l})_{l = 1}^{T}},
\]
where the map \( \lambda \) is defined as in \eqref{eq:definition-of-lambda-a}-\eqref{eq:definition-of-lambda}.
Consider the optimization problem \eqref{eq:DL-problem}.
\begin{itemize}[label=\(\circ\), leftmargin=*]
\item \eqref{eq:DL-problem} admits an optimal solution consisting of an optimal dictionary \( ( \dict{i} \opt )_{i = 1}^{\dsize} \) and an optimal scheme \( \scheme \opt (\cdot) \). 

\begin{itemize}[label=\(\triangleright\), leftmargin=*]
\item A dictionary \( ( \dict{i} \opt )_{i = 1}^{\dsize} \) is an \( \ell_2 \)-optimal dictionary if and only if it satisfies:
\begin{equation}
\sum_{i = 1}^{\dsize} \dict{i} \opt { \dict{i} \opt } \transp = M \opt \Let \sum_{i = 1}^{\dimension} \eigval{i} \opt \; \eigvecsig{i} \eigvecsig{i} \transp,
\end{equation}
for some sequence of orthonormal eigenvectors \( (u\opt_i)_{i = 1}^{\dimension} \) of \( \sigmat{\randomvec} \) corresponding to eigenvalues \( (\eigvalsig{i})_{i = 1}^{\dimension} \).\footnote{It should be noted that multiple such sequences exist when there are multiple eigenvalues which are equal.}

\item The unique optimal scheme \( \scheme\opt (\cdot) \) corresponding to an optimal dictionary \( (\dict{i} \opt)_{i = 1}^{\dsize} \) is given by 
\[ 
\scheme\opt (v) \Let \pmat{ \dict{1}^*}{ \dict{2}^*}{ \dict{\dsize}^*}^+ v .
\]
\end{itemize}

\item The optimal value \( p \opt \) is given by
\[
 - q \opt = p \opt = \optval{(\sqrt{\eigvalsig{i}})_{i = 1}^{\dimension}}{(\length{i}')_{i = 1}^{\dimension}}{(\partition{l})_{l = 1}^{T}},
\]
where the map \( J \) is defined as in \eqref{eq:J-map-definition}. 
\end{itemize}

\end{theorem}

%===============================================================================
\subsection{Step II: \( \const \neq 0 \) and \(X_{\randomvec} (\const) \) is a strict nontrivial subspace of \( \R^{\dimension} \). }

Let \( X_{\randomvec} (\const) \) be any proper nontrivial subspace of \( \R^{\dimension} \). In this situation it is reasonable to expect that no optimal dictionary that solves \eqref{eq:DL-problem-general} contains elements that do not belong to \( X_{\randomvec} (\const) \). That this indeed happens is the assertion of the following Lemma, whose proof is provided in Section \ref{subsection:proof-of-lemma2}:
\begin{lemma}
	\label{lemma 2}
For every \( \const \neq 0 \), optimal solutions of problem \eqref{eq:DL-problem-general}, if any exists, are such that the optimal dictionary vectors \( ( \dict{i}\opt )_{i = 1}^{\dsize} \) satisfy \( \dict{i}\opt \in X_{\randomvec} (\const) \) for all \( i = \indexes{\dsize} \).
\end{lemma}

Lemma \ref{lemma 2} allows us to replace the constraint that \( \Span ( \dict{i} )_{i = 1}^{\dsize} \supset R_{\randomvec} (\const) \) with \( \Span ( \dict{i} )_{i = 1}^{\dsize} = X_{\randomvec} (\const) \) without changing the optimum value (if it admits a solution). We are now in a position to give a complete solution to the \( \ell_2 \)-optimal dictionary problem. 

\begin{theorem}
\label{theorem:DL-theorem-2}
Let \( \mu \Let \EE_{\PP}[V] \), \( \sigmat{}\opt \Let \var(V) \) and \( (\eigvalsig{i} \opt)_{i = 1}^{\poseig} \) be the sequence of positive (non-zero) eigenvalues of \( \sigmat{}\opt \) arranged in non-increasing order. Let the sequence of positive real numbers \( (\length{i}')_{i = 1}^{\poseig} \) be defined by \( ( \length{i}' )_{i = 1}^{\poseig - 1} \Let ( \length{i} )_{i = 1}^{\poseig - 1} \) and \( \length{\poseig}' \Let \sum_{j = \poseig}^{\dsize} \length{j} \).
\begin{enumerate}[leftmargin=*, label = {\rm (\roman*)}, align=left, widest=ii]
\item The optimization problem:
\begin{equation}
\label{eq:QP-in-theorem-2}
\begin{aligned}
		& \minimize_{( \dualvar{t} )_t \subset \R}	&&  \sum_{t = 1}^{\poseig}  \length{t}' \dualvar{t}^2 - 2  \sqrt{\eigvalsig{t}} \dualvar{t}  \\
		& \sbjto						&&
		0 \leq \dualvar{1} \leq \cdots \leq \dualvar{\poseig},
\end{aligned}
\end{equation}
admits a unique optimal solution \( (\dualvar{t} \opt)_{t = 1}^{\poseig} \). 

\item Consider the \( \ell_2 \)-optimal dictionary problem \eqref{eq:DL-problem-general}. Iteratively define an ordered set \( (\partition{1}, \partition{2}, \ldots, \partition{L}) \subset (\indexes{\poseig}) \) by
        \begin{equation}
        \begin{aligned}
        \label{eq:partition-in-theorem}
        \partition{1} & \Let 1, \\
        \partition{l} & \Let \min \{ t \; \vert \; \partition{(l - 1)} < t \leq \dimension, \; \dualvar{(t - 1)} \opt < \dualvar{t} \opt \} \quad  \text{for all \( l = 2,\ldots, L \),} 
        \end{aligned}
        \end{equation}
and let \( (\eigval{i} \opt)_{i = 1}^{\poseig} \) be the non-increasing sequence of positive real numbers defined by
\[
(\eigval{i} \opt)_{i = 1}^{\poseig} \Let \eigvalmap{(\sqrt{\eigvalsig{}\opt})_{i = 1}^{\poseig}}{(\length{i}')_{i = 1}^{\poseig}}{(\partition{l})_{l = 1}^{L}},
\]
where \( \lambda \) is defined as in \eqref{eq:definition-of-lambda-a}-\eqref{eq:definition-of-lambda}.
\begin{itemize}[label=\(\circ\), leftmargin=*]
\item \eqref{eq:DL-problem-general} admits an optimal solution consisting of \(\const\opt \in \R^{\dimension} \), \( (\dict{i}\opt)_{i = 1}^{\dsize} \) and \( \optscheme : \R^{\dimension} \longrightarrow \R^{\dsize} \), satisfying:

\begin{itemize}[label=\(\triangleright\), leftmargin=*]
\item A dictionary \( ( \dict{i} \opt )_{i = 1}^{\dsize} \) is an \( \ell_2 \)-optimal dictionary if and only if it satisfies:
\begin{equation}
\sum_{i = 1}^{\dsize} \dict{i} \opt { \dict{i} \opt } \transp = M\opt \Let \sum_{i = 1}^{\poseig} \eigval{i}\opt \; \eigvecsig{i} \eigvecsig{i} \transp,
\end{equation}
for some sequence of orthonormal eigenvectors \( (u\opt_i)_{i = 1}^{\poseig} \) of \( \sigmat{}\opt \) corresponding to the eigenvalues \( (\eigvalsig{i}\opt)_{i = 1}^{\poseig} \).

\item The unique optimal scheme \( \scheme\opt (\cdot) \) corresponding to an optimal dictionary \( (\dict{i} \opt)_{i = 1}^{\dsize} \) is given by 
\[ 
\scheme\opt (v) \Let \pmat{ \dict{1}^*}{ \dict{2}^*}{ \dict{\dsize}^*}^+ v .
\]
\end{itemize}

\item The optimal value \( p \opt \) of \eqref{eq:DL-problem-general} is given by
\[
 p \opt = \optval{(\sqrt{\eigvalsig{i}})_{i = 1}^{\poseig}}{(\length{i}')_{i = 1}^{\poseig}}{(\partition{l})_{l = 1}^{L}} = - q \opt,
\]
where \( q\opt \) is the value of \eqref{eq:QP-in-theorem-2} and the mapping \( J \) is defined as in \eqref{eq:J-map-definition}.
\item Moreover, every optimal dictionary in \eqref{eq:DL-problem-general} can be computed via Algorithm \ref{algo:general-l2-opt-algo} in polynomial time. 
\end{itemize}
\end{enumerate}
\end{theorem}

\begin{corollary}
\label{corollary:equal-length-dictionary}
Consider the \( \ell_2 \)-optimal dictionary problem \eqref{eq:DL-problem-general} with its associated data. If \( \length{i} = 1 \) for all \( i = \indexes{\dsize} \), then a dictionary \( (\dict{i}\opt)_{i = 1}^{\dsize} \subset \R^{\dimension} \) is \( \ell_2 \)-optimal if and only if it satisfies
\[
\sum_{i = 1}^{\dsize} \dict{i}\opt {\dict{i}\opt}\transp = M\opt \Let \frac{K}{\trace({\sigmat{}\opt}^{1/2})} {\sigmat{}^*}^{1/2}.
\]
\end{corollary}

\begin{remark}
Corollary \ref{corollary:equal-length-dictionary} is the main result of \cite[Theorem 2]{sheriff2016optimal}.
\end{remark}

\begin{algorithm}
\label{algo:general-l2-opt-algo}
\KwIn{The variance matrix $\sigmat{}\opt \in \possemdef{\dimension}$ and a number $\dsize \geq \poseig \Let \rank(\sigmat{}\opt)$.}
\KwOut{An \( \ell_2 \)-optimal dictionary-scheme pair \( \bigl( ( \dict{i}\opt )_{i = 1}^{\dsize}, \optscheme \bigr) \) that solves \eqref{eq:DL-problem-general}.}

\nl Compute the sequence of orthonormal eigenvectors \( (u_i)_{i = 1}^{\poseig} \) of \( \sigmat{}\opt \) corresponding to the eigenvalues \( (\eigvalsig{i}\opt)_{i = 1}^{\poseig} \).

\nl Solve the QP \eqref{eq:QP-in-theorem-2}, from the optimal solution \( (\dualvar{t} \opt)_{t = 1}^{\poseig} \) compute \( (\partition{1}, \partition{2}, \ldots, \partition{L}) \subset (\indexes{\poseig}) \) and \( (\eigval{i}\opt)_{i = 1}^{\poseig} \) as asserted by Theorem \ref{theorem:DL-theorem-2}.

\nl Define \( M\opt \Let \sum\limits_{i = 1}^{\poseig} \eigval{i}\opt u_i u_i\transp \), and get a rank-\( 1 \) decomposition \( (\dict{i}\opt)_{i = 1}^{\dsize} \) of \( M\opt \) from Algorithm \ref{algo: rank-1 decomposition} using \( M\opt \) and \( (\length{i})_{i = 1}^{\dsize} \) as inputs.

\nl Output th optimal dictionary \( (\dict{i}\opt)_{i = 1}^{\dsize} \) and the optimal scheme \( \optscheme (v) \Let { \pmat{\dict{1}\opt}{\dict{2}\opt}{\dict{\dsize}\opt} }^+ v. \)

\caption{A procedure to obtain $\ell_2$-optimal dictionaries.}

\end{algorithm}

\subsection{An example in \( \R^2 \)}
	\label{ex:1}
Let \( R_1 \), \( R_2 \) and \( R_3 \) be subsets of \( \R^2 \) given by
\begin{align*}
R_1 &\Let \{ (x \ y)\transp \in \R^2 \vert -2 \leq x \leq 0, \ 1 \leq y \leq 3 \}, \\
R_2 &\Let \{ (x \ y)\transp \in \R^2 \vert -6 \leq x \leq -2, \ -4 \leq y \leq -2 \}, \quad \text{and} \\
R_3 &\Let \{ (x \ y)\transp \in \R^2 \vert 2 \leq x \leq 4, \ 0 \leq y \leq 2 \}. 
\end{align*}
Let \( \randomvec \) be a random vector taking values in \( \R^2 \) distributed according to the density \( \rho_V \) given by
\[
\R^2 \ni v \longmapsto \rho_V (v) \Let \frac{1}{8} \mathds{1}_{R_1}(v) + \frac{1}{32} \mathds{1}_{R_2}(v) + \frac{1}{16} \mathds{1}_{R_3}(v) \in [0,1].
\]
From elementary calculations we find that the mean \( \mu \) and the variance \( \sigmat{}\opt \) of \( \randomvec \) are equal to \( \pmatr{-3/4\\ 1/2} \) and \(
\pmatr{6.7708 & 3.1250 \\ 3.1250 & 4.5833} \), respectively.

Suppose that our objective is to represent samples of \( \randomvec \) using an \( \ell_2 \)-optimal dictionary  as in \eqref{eq:DL-problem-general}. We consider the case of finding an \( \ell_2 \)-optimal dictionary of three vectors, with \( \length{1} = 2 \), \( \length{2} = 1 \) and \( \length{3} = 1 \).\footnote{The choice of \( \length{i} \)s in the examples is arbitrary.} We compute an \( \ell_2 \)-optimal dictionary \( \dictionary\opt = (\dict{1}\opt, \dict{2}\opt, \dict{3}\opt ) \) using Algorithm \ref{algo:general-l2-opt-algo}; the dictionary vectors are
\[
\dict{1}\opt = \pmatr{-0.4611\\ -1.3369}, \ \ \dict{2}\opt = \pmatr{1 \\ 0} \text{and} \ \ \dict{3}\opt = \pmatr{-1\\ 0},
\]
and the optimal \( \ell_2 \) cost of representation is 5.1444. The \( \ell_2 \)-optimal dictionary \( \dictionary\opt \) along with the distribution of \( \randomvec \) is depicted in Figure \ref{fig:2d-example}.
\begin{figure}[H]
\begin{center}
\begin{tikzpicture}[%
    scale=1,%
    IS/.style={blue, thick},%
    LM/.style={red, thick},%
    axis/.style={very thick, ->, >=stealth', line join=miter},%
    important line/.style={thick}, dashed line/.style={dashed, thin},%
    every node/.style={color=black},%
    dot/.style={circle,fill=black,minimum size=4pt,inner sep=0pt,%
        outer sep=-1pt},%
]

\shade[top color=black!60, bottom color=black!60,shading angle=0,line width=0.1mm] (-2,1) rectangle (0,3);
\coordinate (R1) at (-1,2);
\draw [fill=red] (R1) node {\( R_1 \)};

\shade[top color=black!25, bottom color=black!25,shading angle=0,line width=0.1mm] (-6,-4) rectangle (-2,-2);
\coordinate (R2) at (-4,-3);
\draw [fill=red] (R2) node {\( R_2 \)};

\shade[top color=black!40, bottom color=black!40,shading angle=0,line width=0.1mm] (2,0) rectangle (4,2);
\coordinate (R3) at (3,1);
\draw [fill=red] (R3) node {\( R_3 \)};

\coordinate (mu) at (-0.75,0.5);
\draw [fill=red] (mu)  circle(1pt) node [above right] {\( \mu \)};

\coordinate (d1) at (-0.4611-0.75,-1.3369+0.5);
\draw [fill=red] (d1) circle(1.5pt) node [below left] {\( d_1^* \)};
\draw [->, blue, thick] (mu) -- (d1);

\coordinate (d2) at (-1-0.75,0+0.5);
\draw [fill=red] (d2) circle(1.5pt) node [below left] {\( d_2^* \)};
\draw [->, blue, thick] (mu) -- (d2);

\coordinate (d3) at (1-0.75,0+0.5);
\draw [fill=red] (d3) circle(1.5pt) node [below right] {\( d_3^* \)};
\draw [->, blue, thick] (mu) -- (d3);

\draw[dashed] (0.366-0.75,-1.366+0.5) arc (-75:-125:1.4142);
\draw[dashed] (-1-0.75,0+0.5) arc (-180:0:1);

\coordinate (dummy1) at (0.1-0.75,-0.4619+0.5);
\draw [fill=red] (dummy1) node [below, scale = 0.55] {\( \sqrt{2} \)};
\draw [->, black, thin] (dummy1) -- (mu);
\coordinate (dummy2) at (0.15-0.75,-0.6928+0.5);
\coordinate (dummy3) at (0.3-0.75,-1.3856+0.5);
\draw [->, black, thin] (dummy2) -- (dummy3);

\coordinate (dummy3) at (0.7071-0.75,-0.7071+0.5);
\coordinate (dummy1) at (0.2828-0.75,-0.2828+0.5);
\coordinate (dummy2) at (0.4596-0.75,-0.4596+0.5);
\draw [fill=red] (dummy1) node [below right, scale = 0.55] {\( 1 \)};
\draw [->, black, thin] (dummy1) -- (mu);
\draw [->, black, thin] (dummy2) -- (dummy3);

\draw[->, dotted] (-6,0.5)--(4.25,0.5) node[right]{$x$};
\draw[->, dotted] (-0.75,-4.5)--(-0.75,3.75) node[above]{$y$};
\end{tikzpicture}
\caption{Example of an \( \ell_2 \)-optimal dictionary problem in \( \R^2 \).}
\label{fig:2d-example}
\end{center}
\end{figure}
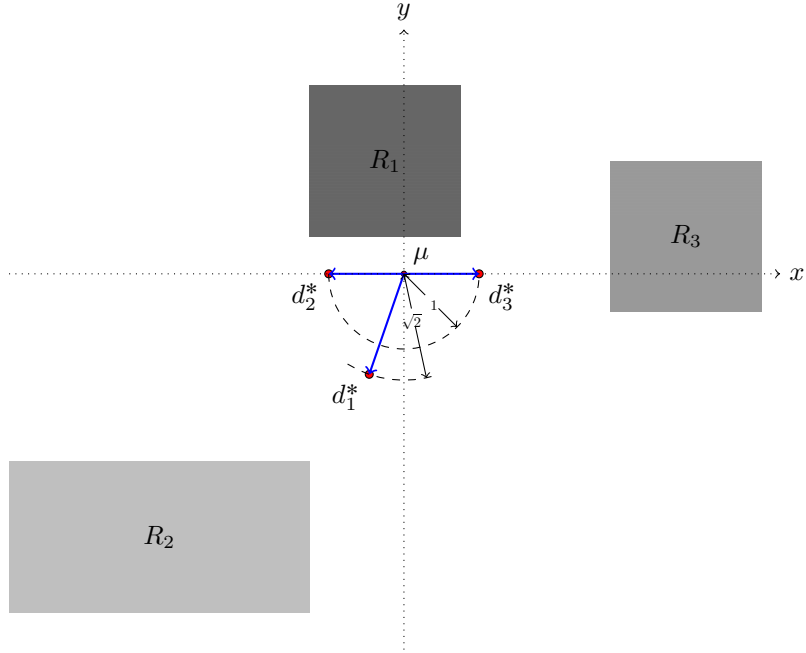

\subsection{An example in \( \R^3 \)}

Let \( \sigmat{} \in \posdef{3} \) and let \( Y \sim \mathcal{N}(0, \sigmat{}) \). Then \( V \Let \frac{Y}{\norm{Y}} \) is a random vector taking values on the unit sphere of \( \R^3 \), it is uniformly distributed when \( \sigmat{} = I_3 \) but not otherwise. We consider the \( \ell_2 \)-optimal dictionary problem \eqref{eq:DL-problem-general} for optimal representation of \( V \) using  dictionaries of four vectors. For \( \sigmat{} = \pmatr{3.5 & -2.5 & -2.1213 \\ -2.5 & 3.5 & 2.1213 \\ -2.1213 & 2.1213 & 6} \) a quick calculation gives 
\[
\var (\randomvec) = \pmatr{0.2855 & -0.1445 & -0.0919 \\ -0.1445 & 0.2855 & 0.0919 \\ -0.0919 & 0.0919 & 0.4300}. 
\]
An \( \ell_2 \)-optimal dictionary \( \dictionary\opt = (\dict{1}\opt, \dict{2}\opt, \dict{3}\opt, \dict{4}\opt) \) is computed via Algorithm \ref{algo:general-l2-opt-algo} for \( \length{1} = 8 \), \( \length{2} = 4 \), \( \length{3} = 2 \) and \( \length{4} = 1 \) and \( \sigmat{}\opt \Let \var (\randomvec) \); and the resulting dictionary vectors are:
\[
\dict{1}\opt = \pmatr{-\sqrt{2} \\ \sqrt{2} \\ 2}, \ \dict{2}\opt = \pmatr{1 \\ -1 \\ \sqrt{2}}, \ \dict{3}\opt = \pmatr{1 \\ 1 \\ 0}, \text{ and } \dict{4}\opt = \pmatr{\frac{-1}{\sqrt{2}} \\ \frac{-1}{\sqrt{2}} \\ 0}. 
\]

It can be easily verified that the eigenvalues of \( \sigmat{}\opt \) are \( 0.56 \), \( 0.3 \) and \( 0.1410 \) and \( u_1 = (\frac{-1}{2} \ \frac{1}{2} \ \frac{1}{\sqrt{2}})\transp \), \( u_2 = (\frac{-1}{2} \ \frac{1}{2} \ \frac{-1}{\sqrt{2}})\transp \) and \( u_3 = (\frac{1}{\sqrt{2}} \ \frac{1}{\sqrt{2}} \ 0)\transp \) are the corresponding eigenvectors. It is to be noted that the dominant dictionary vectors \( \dict{1}\opt \) and \( \dict{2}\opt \) are oriented towards the ``most probable'' directions, i.e., the directions of \( u_1 \) and \( u_2 \) respectively. 

\subsection{Uniform distribution over the unit sphere}
\label{uniform section}
We shall test our results on the important case of \( \mathsf{P} \) being the uniform distribution on the unit sphere. Note that due to (rigid) rotational symmetry of the distribution, it follows that rigid rotations of optimal dictionaries in this case are also optimal. Moreover, in the case of \( \mathsf{P} \) being uniform distribution over unit sphere, no direction in the space is prioritized over other. In such a case, it is expected that the \( \ell_2 \)-optimal dictionaries are maximally spread in the space. Such maximally spread out dictionaries are formally studied as \emph{tight frames} in the theory of frames.

We recall here some standard definitions for completeness and to provide the necessary substratum for our next result. Let $ \dimension, \dsize $ be positive integers such that $ \dsize \geq \dimension $. We say that a collection of vectors \( ( x_i )_{i = 1}^{\dsize} \) is a \emph{frame} for \( \R^{\dimension} \) if there exist some constants \( c, C > 0 \) such that
\[
	c \norm{x}^2 \leq \sum_{i = 1}^K \lvert \inprod{x_i}{x}^2 \leq C \norm{x} \rvert^2 \quad \text{for all $x \in \R^n$.}
\]
We say that a frame \( (x_i)_{i = 1}^{\dsize}\subset\R^{\dimension} \) is \emph{tight} if $c = C$. We have the following connection between \(\ell_2\)-optimal dictionaries and tight frames:
\begin{proposition}
\label{proposition:tight-frame-and-opt-dictionaries}
Consider the \( \ell_2 \)-optimal dictionary problem \eqref{eq:DL-problem-general} with \( \mathsf{P} \) being the uniform distribution over the unit sphere. If \( \length{1} \leq \frac{1}{\dimension} \sum_{i = 1}^{\dsize} \length{i} \), then a feasible dictionary is \( \ell_2 \)-optimal if and only if it is a tight frame of \( \R^{\dimension} \).
\end{proposition}

\subsection{Robustness of the \( \ell_2 \)-optimal dictionaries with respect to the estimation of the mean and variance of the distribution \( \mathsf{P} \).}
In practice, even though the mean and variance of a distribution can be estimated empirically from data to arbitrary precision, specifying their precise values is a difficult matter.Therefore, it is of  practical importance that the dictionary computed using the estimated distribution parameters performs well. To wit, if the estimation of the distribution parameters is good, the performance of the corresponding dictionary should be near optimal. To ensure this, we need continuity of the \( \ell_2 \)-optimal dictionary as a function of mean and variance of the distribution, a property that we shall now study.

Let \( \mu' \), \( \sigmat{}' \) be the estimated mean and variance of the distribution using sufficiently large number of samples drawn from \( \mathsf{P} \).  We have the following continuity result.

\begin{proposition}
\label{proposition:robustness-wrt-estimates}
Let \(  D(\mu', \sigmat{}') \Let \pmat{\dict{1}'}{\dict{2}'}{\dict{\dsize}'} \) be the dictionary computed via Algorithm \ref{algo:general-l2-opt-algo} using the estimated mean and variance. Let \( \scheme'(V) \) be the corresponding scheme of representation. Then the cost of representation \( J(\mu', \sigmat{}') \Let \EE_{\mathsf{P}} [\norm{\scheme'(V)}^2] \), of representing \( \randomvec \) with a constant of representation \( \mu' \) using the dictionary \( D(\mu', \sigmat{}') \) satisfies
\[
J(\mu', \sigmat{}') \xrightarrow[ \substack{ \mu' \rightarrow \mu, \\ \sigmat{}' \rightarrow \sigmat{}\opt }]{} J(\mu, \sigmat{}\opt),
\]
where \( \mu = \EE_{\mathsf{P}}[V] \) and \( \sigmat{}\opt = \var (V) \).
\end{proposition}

%%%%%%%%%%%%%%%%%%%%%%%%%%%%%%%%%%%%%%%%%%%%%%%%%%%%%%%%%%%%%%%%%%%%%%%%%%%%%%%%%%%%%%%%%%%%%%%%%%
\section{Mathematical tools and other auxiliary results}
\label{section:Mathematical tools and other auxiliary results}
We shall discuss and establish some mathematical results in this section, which are essential for solving the \( \ell_2 \)-optimal dictionary problem.

\subsection{Majorization and rank-1 decomposition}
\label{subsection: majorization}
The theory of majorization, sometimes also referred to as Schur convexity, plays a crucial role in solving problem \( \ell_2 \)-optimal dictionary problem, together with a specific class of rank-\(1 \) decomposition of positive operators. We start with majorization: 

\begin{definition}
\label{definition: majorization}
Let \( \dsize \) be a positive integer and \( (a_i)_{i = 1}^{\dsize} \), \( (b_i)_{i = 1}^{\dsize} \) be two non-increasing sequences of real numbers. We say that the sequence \( (a_i)_{i = 1}^{\dsize} \) is \emph{majorized} by the sequence \( (b_i)_{i = 1}^{\dsize} \), written \( (a_i)_{i = 1}^{\dsize} \majorize (b_i)_{i = 1}^{\dsize} \),  if and only if 
\begin{equation}
\label{eq: majorization}
\begin{aligned}
        \sum_{i = 1}^\dsize a_i &= \sum_{i = 1}^\dsize b_i \; , \text{and}\\
        \sum_{i = 1}^m a_i &\leq \sum_{i = 1}^m b_i \quad \text{for every \( m = \indexes{( \dsize - 1 )} \).}
\end{aligned}
\end{equation}
\end{definition}

Recall that a map \( \pi : \{ \indexes{\dsize} \} \longrightarrow \{ \indexes{\dsize} \} \), is a called a \emph{permutation} map if it injective and let \( \Pi_{\dsize} \) be the set of all permutation maps on \( \{ \indexes{\dsize} \} \). The \emph{permutation polytope} \( \permpolytope{ (b_i)_{i = 1}^{\dsize}} \) of a sequence \( (b_i)_{i = 1}^{\dsize} \) is defined as the convex hull of  \( \{ ( b_{\pi(1)}, b_{\pi(2)}, \ldots, b_{\pi(\dsize)} )^{\transp} \vert \pi \in \Pi_{\dsize} \} \) in \( \R^{\dsize} \). The following two classical results in the theory of majorization are essential for us.

\begin{lemma}\cite[Lemma 5]{kadison2002pythagorean}
\label{lemma: Kadison}
Let \( K \) be a positive integer, \( (a_i)_{i = 1}^{\dsize} \) and \( (b_i)_{i = 1}^{\dsize} \) be two non-increasing sequences of real numbers and let \( \permpolytope{ (b_i)_{i = 1}^{\dsize}} \) be the permutation polytope of the sequence \( (b_i)_{i = 1}^{\dsize} \). Then the following statements are equivalent
\begin{itemize}[label=\(\circ\), leftmargin=*]
      \item \( (a_i)_{i = 1}^{\dsize} \in \permpolytope{ (b_i)_{i = 1}^{\dsize}} \),
      \item \( (a_i)_{i = 1}^{\dsize} \majorize (b_i)_{i = 1}^{\dsize} \).
\end{itemize}
\end{lemma}

\begin{theorem}\cite[Theorem 5]{Schurhorn}(Schur-Horn theorem)
 \label{theorem: schur horn}
Let \( \dsize \) be a positive integer, and let \( (\length{i})_{i = 1}^\dsize \) and \( (\lambda_i)_{i = 1}^\dsize \) be two non-increasing sequences of real numbers. Then the following statements are equivalent
\begin{itemize}[label=\(\circ\), leftmargin=*]
\item There exists an Hermitian matrix \( M  \in \symmetric{\dsize} \) having eigenvalues \( (\lambda_i)_{i = 1}^\dsize \) and diagonal entries \( (\length{i})_{i = 1}^{\dsize} \).

\item \( (\length{i})_{i = 1}^{\dsize} \majorize (\lambda_i)_{i = 1}^\dsize \). 
\end{itemize}
\end{theorem}

In concern to the \( \ell_2 \)-optimal dictionary problem \eqref{eq:DL-problem-general} we need the following special form of the Schur-Horn theorem inspired from \cite[Proposition 3.1]{casazza2002frames}:  

\begin{lemma}
\label{lemma: orthonormal basis}
Let \( \dsize \) be a positive integer, and let \( (\length{i})_{i = 1}^{\dsize} \) be a non-increasing sequence of real numbers. Let \( \linearmap : \R^{\dsize} \lra \R^{\dsize}\) be a linear map and \( (\lambda_i)_{i = 1}^{\dsize} \) be the non-increasing sequence of the eigenvalues of \( \frac{1}{2} (\linearmap + \linearmap^{\transp}) \). Then the following statements are equivalent
\begin{itemize}[label=\(\circ\), leftmargin=*]
\item  \( (\length{i})_{i = 1}^{\dsize} \majorize (\lambda_i)_{i = 1}^{\dsize} \)

\item There exists an orthonormal basis \( ( \orthobasis{i} )_{i = 1}^\dsize \) to \( \R^\dsize \) such that 
    \[
    \inprod{ \orthobasis{i} }{ \linearmap \orthobasis{i} } = \length{i} \quad \text{for all \(i = \indexes{\dsize} \)}.
    \]
    Moreover, such an orthonormal basis can be obtained from Algorithm \ref{algo: orthonormal basis}.
\end{itemize}  
\end{lemma}

\begin{algorithm}
\label{algo: orthonormal basis}
\caption{Calculation of orthonormal bases \`a la Lemma \ref{lemma: orthonormal basis}}
\KwIn{A linear map \( \linearmap : \R^{\dsize} \lra \R^{\dsize} \) and a non-increasing sequence \( ( \length{i} )_{i = 1}^{\dsize} \) of positive real numbers.}
\KwOut{ An orthonormal collection of vectors \( ( \orthobasis{i} )_{i = 1}^{\dsize} \subset \R^{\dsize} \) such that \( \inprod{ \orthobasis{i} }{ \linearmap \orthobasis{i} } = \length{i} \) for all \( i = \indexes{\dsize} \).}

\nl Compute \( \frac{1}{2} (\linearmap + \linearmap^{\transp}) \).

\nl Compute the eigenvectors \( (\eigvec{i})_{i = 1}^{\dsize} \subset \R^{\dsize} \) of \( \frac{1}{2} (\linearmap  + \linearmap^{\transp}) \).

\nl Initialize the following quantities by 
\begin{align*}
( \uvec{i}(1) )_{i = 1}^{\dsize} & \Let \linearmap \sort \big\{ (\eigvec{i})_{i = 1}^{\dsize} \big\} ,\footnotemark[5] \\
( \aval{i}(1) )_{i = 1}^{\dsize} & \Let (\length{i})_{i = 1}^{\dsize} . 
\end{align*}

\nl \textbf{for} \( t \) from \( 1 \) to \( \dsize \) \textbf{do} 

\quad Define \( \bval{1} \Let \inprod{\uvec{1}(t)}{\uvec{1}(t)}. \)

\quad \textbf{case 1} : If \( \bval{1} = \aval{1}(t) \),

\quad \quad update
\begin{equation}
\begin{aligned}
\label{eq: case 1 iteration update}
\orthobasis{t} \; & \Let \; \uvec{1}(t), \\
\big( \uvec{i}(t + 1) \big)_{i = 1}^{(\dsize - t)} \; & \Let \; \linearmap \sort  \big\{ (\uvec{(i + 1)}(t))_{i = 1}^{( \dsize - t )} \big\}, \\
(\aval{i}(t + 1))_{i = 1}^{(\dsize - t)} \; & \Let \; (\aval{(i + 1)}(t))_{i = 1}^{(\dsize - t)} .
\end{aligned}
\end{equation}

\quad \textbf{case 2} : If \( \bval{1} > \aval{1}(t) \),

\quad \quad find \( i \) such that \( 1 < i \leq (\dsize - t + 1) \), and
\[
\bval{i} \Let \inprod{\uvec{i}(t)}{\linearmap \uvec{i}(t)} \quad \leq \; \aval{1}(t) \; < \quad \inprod{\uvec{(i - 1)}(t)}{\linearmap \uvec{(i - 1)}(t)}.
\]

\quad \quad Define, 
\begin{equation}
\begin{aligned}
\label{eq: defining theta, x ,v}
\Theta & \Let \frac{\sqrt{\aval{1}(t) - \bval{i}}}{\sqrt{\aval{1}(t) - \bval{i}} + \sqrt{\bval{1} - \aval{1}(t)} } , \\
v & \Let \frac{(1 - \Theta) \uvec{1}(t) - \Theta \uvec{i}(t)}{\sqrt{\Theta^2 + (1 - \Theta)^2}}.
\end{aligned}
\end{equation} 

\quad \quad Update,
\begin{equation}
\begin{aligned}
\label{eq: case 2 iteration update}
\orthobasis{t} \; & \Let \; \frac{\Theta \uvec{1}(t) + (1 - \Theta) \uvec{i}(t)}{\sqrt{\Theta^2 + (1 - \Theta)^2}}, \\
\big( \uvec{i}(t + 1) \big)_{i = 1}^{(\dsize - t)} \; & \Let \; \linearmap \sort  \big\{ \{ \uvec{i}(t) \}_{i = 1}^{( \dsize - t + 1 )} \; \setminus \{ \uvec{1}(t), \uvec{i}(t) \} \; \cup \; \{ v \} \big\}, \\
(\aval{i}(t + 1))_{i = 1}^{(\dsize - t)} \; & \Let \; (\aval{(i + 1)}(t))_{i = 1}^{(\dsize - t)} .
\end{aligned}
\end{equation}

\nl \textbf{end for loop}

\nl Output \( (\orthobasis{t})_{t = 1}^{\dsize} \).
\end{algorithm}
\footnotetext[5]{Please see \eqref{eq:definition-of-Asort} for the definition of \( \linearmap \sort \).}

The main outcome of Lemma \ref{lemma: orthonormal basis} is a specific class of rank-\(1\) decompositions of non negative definite matrices which is used directly in solving \eqref{eq:DL-problem-general}. While these facts will be essential to establish our main results, they are also of independent interest.

It is a standard result in the theory of matrices \cite[p.\ 2]{bhatiapositive}, that a non-negative definite matrix \( M \in \posdef{n} \) with real entries can be decomposed as 
\begin{equation}
\label{decomposition}
   	M = \sum_{i = 1}^{\dsize} \rodvec{i} \rodvec{i}^{\transp},
\end{equation}
where \( \dsize \geq r \Let \rank(M) \), and \( (\rodvec{i})_{i = 1}^{\dsize} \) is a sequence of vectors in \( \R^n \). Since \( \rodvec{i} \rodvec{i}^{\transp} \) is a matrix of rank one and \( M \) is expressed as a sum of such matrices each of rank one, we say that \eqref{decomposition} is a \emph{rank-1 decomposition} of \( M \) into the sequence \( (\rodvec{i})_{i = 1}^{\dsize} \).

There are numerous rank-1 decompositions of non-negative definite matrices, each fine tuned for some specific purpose; see e.g., \cite[Theorem 7.3]{zhangmatrix}. With respect to the \( \ell_2 \)-optimal dictionary problem \eqref{eq:DL-problem-general} and its connections to rank-1 decomposition of non-negative definite matrices, we need an answer to the following question: Suppose that a non-negative definite matrix \( M \)  and a non-increasing sequence \( (\length{i})_{i = 1}^{\dsize} \) of positive real numbers are given. Does there exist a rank-1 decomposition of  \( M \) as \( M = \sum_{i = 1}^{\dsize} \rodvec{i} \rodvec{i}\transp \) for some vectors \( (\rodvec{i})_{i = 1}^{\dsize} \subset \R^{\dimension} \) satisfying 
\[
\inprod{\rodvec{i}}{\rodvec{i}} = \length{i} \quad \text{for all \( i = \indexes{\dsize}  \)}?
\]
This question was answered using the Schur-Horn theorem \cite[Theorem 2.1]{casazza2002frames} by providing necessary and sufficient conditions for the existence of such a rank-1 decomposition. We restate this result for easy reference:
\begin{theorem}\cite[Theorem 2.1]{casazza2002frames}
\label{theorem: casazza rod}
Let \( \dimension \) be a positive integer, \( M \in \possemdef{n} \) and \( ( \length{i} )_{i = 1}^{\dsize} \) be a non-increasing sequence of positive real numbers with \( \dsize \geq r \Let \rank(M)\). Let \( ( \eigval{i} )_{i = 1}^r \) be the sequence of the nonzero eigenvalues of \( M \) arranged in non-increasing order. Then the following statements are equivalent:
\begin{itemize}[label = \( \circ \), leftmargin=*]
\item There exists a rank-1 decomposition of \( M \) into a sequence \( (\rodvec{i})_{i = 1}^{\dsize} \) that satisfies 
\begin{equation}
\label{equation: rank-1 decomposition lengths}
   \inprod{\rodvec{i}}{\rodvec{i}} = \length{i} \quad \text{for all \( i = \indexes{\dsize}, \)}.
\end{equation}

\item The sequences \( ( \length{i} )_{i = 1}^{\dsize} \) and \( ( \eigval{i} )_{i = 1}^r \) satisfy
\begin{equation}
\label{eq:rod-majorization-condition}
 \begin{aligned}
        \sum_{i = 1}^{\dsize} \length{i} &= \sum_{i = 1}^r \lambda_i , \\
        \sum_{i = 1}^m \length{i} &\leq \sum_{i = 1}^m \lambda_i \quad \text{for every \( m = \indexes{r - 1} \).}
    \end{aligned}
\end{equation}

\end{itemize}
 \end{theorem}

Given \( M \in \possemdef{\dimension} \) with nonzero eigenvalues \( (\eigval{i})_{i = 1}^{r} \) and corresponding eigenvectors \( ( \eigvec{i} )_{i = 1}^{r} \), the spectral theorem  shows that we can write 
\[
M = \sum_{i = 1}^{r} \eigval{i} \eigvec{i} \eigvec{i}^{\transp}.
\]
By defining \( C \in \R^{\dimension \times \dsize} \) as
\[
C \Let \big( \sqrt{\eigval{1}} ~ \eigvec{1} \quad \sqrt{\eigval{2}} ~ \eigvec{2} \; \cdots \; \sqrt{\eigval{r}} ~ \eigvec{r} \quad 0 \; \cdots \; 0  \big),
\]
we see at once that \( M = C C\transp \). Let a linear map \( \Lambda : \R^{\dsize} \lra \R^{\dsize} \) be defined as
\begin{equation*}
\Lambda \Let C^{\transp} C = \diag(\eigval{1}, \eigval{2}, \ldots, \eigval{r}, 0, \ldots, 0),
\end{equation*}
and let \( (\length{i})_{i = 1}^{\dsize} \) be a non-increasing sequence satisfying the hypothesis of Theorem \ref{theorem: casazza rod}; it can then be observed that \( (\length{i})_{i = 1}^{\dsize} \majorize (\eigval{1}, \eigval{2}, \ldots, \eigval{r}, 0, \ldots, 0 ) \). From Lemma \ref{lemma: orthonormal basis} we know that an orthonormal sequence of vectors \( (\orthobasis{i})_{i = 1}^{\dsize} \subset \R^{\dsize} \) exists such that 
\[
\inprod{\orthobasis{i}}{\Lambda \orthobasis{i}} = \length{i} \quad \text{for all \( i = \indexes{\dsize}, \)}
\]
and \( (\orthobasis{i})_{i = 1}^{\dsize} \) can be obtained via Algorithm \ref{algo: orthonormal basis} with \( \Lambda \) and \( (\length{i})_{i = 1}^\dsize \) as inputs. Let us define a sequence of vectors \( (\rodvec{i})_{i = 1}^{\dsize} \subset \R^{\dimension} \) by 
\[
\rodvec{i} \Let C x_i \quad \text{for all \( i = \indexes{\dsize} \)}.
\]
It follows readily that 
\begin{equation*}
\inprod{\rodvec{i}}{\rodvec{i}} = \length{i} \quad \text{ for all \( i = \indexes{\dsize} \), and } \sum_{i = 1}^{\dsize} y_i y_i\transp = M.
\end{equation*}

Thus, \( (\rodvec{i})_{i = 1}^{\dsize} \) is a rank-1 decomposition of \( M \). The above analysis is summarized in the form of Algorithm \ref{algo: rank-1 decomposition} that gives a procedure to obtain the above rank-1 decomposition of non negative matrices.

\begin{algorithm}[h]
\label{algo: rank-1 decomposition}
\caption{Computation of a rank-1 decomposition of non negative definite matrices \`a la Theorem \ref{theorem: casazza rod}}
\KwIn{A non negative definite matrix \( M \in \possemdef{\dimension} \) and a non-increasing sequence \( ( \length{i} )_{i = 1}^{\dsize} \) of positive real numbers.}
\KwOut{ A rank-1 decomposition \( (\rodvec{i})_{i = 1}^{\dsize} \) of \( M \) `a la Theorem \ref{theorem: casazza rod}.}

\nl Compute eigen-decomposition of \( M \) to get nonzero eigenvalues \footnotemark[6]  \( (\eigval{i})_{i = 1}^{r} \) and the corresponding eigenvectors \( (\eigvec{i})_{i = 1}^{r} \subset \R^{\dimension}\). 

\nl Define 
\begin{align*}
  \R^{\dsize \times \dsize} \ni \Lambda & \Let \diag(\eigval{1}, \eigval{2}, \ldots, \eigval{r}, 0, \ldots, 0 ), \\
  \R^{\dimension \times \dsize} \ni C & \Let \big( \sqrt{\eigval{1}} ~ \eigvec{1} \quad \sqrt{\eigval{2}} ~ \eigvec{2} \; \cdots \; \sqrt{\eigval{r}} ~ \eigvec{r} \quad 0 \; \cdots \; 0  \big).
\end{align*}

\nl Get the sequence \( (\orthobasis{i})_{i = 1}^{\dsize} \subset \R^{\dsize} \) via Algorithm \ref{algo: orthonormal basis} with \( \Lambda \) and \( ( \length{i} )_{i = 1}^{\dsize} \) as inputs.

\nl Define a sequence \( (\rodvec{i})_{i = 1}^{\dsize} \subset \R^{\dimension} \) as 
\begin{align*}
 \R^{\dimension} \ni \rodvec{i} \Let C \orthobasis{i} \quad \text{fro all \( i = \indexes{\dsize} \).}
\end{align*}

\nl Output \( (\rodvec{i})_{i = 1}^{\dsize} \).

\end{algorithm}
\footnotetext[6]{The eigenvalues are sorted in non-increasing order.}

\begin{remark}
It should be noted that Algorithm \ref{algo:general-l2-opt-algo} uses Algorithm \ref{algo: rank-1 decomposition}, however the eigen-decomposition in step 1 of Algorithm \ref{algo: rank-1 decomposition} can be avoided when called from Algorithm \ref{algo:general-l2-opt-algo}.
\end{remark}

\subsection{Rearrangement inequality}
\label{subsection: rearrangement inequality} We need the following classical rearrangement inequality. Let \( \dimension \) be a positive integer and let \( (a_i)_{i = 1}^{\dimension} \), \( (b_i)_{i = 1}^{\dimension} \) be two arbitrary sequences of non-negative real numbers. Let \( (\overline{a}_i)_{i = 1}^{\dimension} \), \( (\overline{b}_i)_{i = 1}^{\dimension} \) be the rearrangements of the sequences \( (a_i)_{i = 1}^{\dimension} \), \( (b_i)_{i = 1}^{\dimension} \) in non-increasing order and \( (\hat a_i)_{i = 1}^{\dimension} \), \( (\hat b_i)_{i = 1}^{\dimension} \) in non-decreasing order. Then, from \cite[Section 10.2, Theorem 368, p.\ 261]{hardy1952inequalities} we have the following \emph{rearrangement inequality}:
\begin{equation}
\label{rearrangement inequality}
\sum_{i = 1}^{\dimension} \overline{a}_i \hat b_{i} = \sum_{i = 1}^{\dimension} \hat a_i \overline{b}_i \quad
\leq \quad
\sum_{i = 1}^{\dimension} a_i b_i \quad
\leq \quad
\sum_{i = 1}^{\dimension} \overline{a}_i \overline{b}_i = \sum_{i = 1}^{\dimension} \hat a_i \hat b_i.
\end{equation}

\subsection{Auxiliary Optimization problems:}
\label{optimization problem}
Let \( \poseig \) be a positive integer, and let \( (a_i)_{i = 1}^{\poseig}, (s_i)_{i = 1}^{\poseig}  \) be two arbitrary sequences of positive real numbers. Let us consider the following optimization problem:
\begin{equation}
\label{eq: primalgeneral}
\begin{aligned}
		& \minimize_{( \primalvar{i} )_i \subset \R }	&&  \sum_{i = 1}^{\poseig} s_i \Big( \frac{1}{\eigval{i}} \Big) \\
		& \sbjto							&&
		\begin{cases}
            \sum_{i = 1}^j a_i \leq \sum_{i = 1}^j \primalvar{i} \quad \text{ for all \( j = \indexes{\poseig - 1}\)}, \\
			\sum_{i = 1}^{\poseig} a_i = \sum_{i = 1}^{\poseig} \primalvar{i}.
		\end{cases}
	\end{aligned}
\end{equation}
Since \( s_i > 0 \), we know that the map \( \R \ni \primalvar{i} \lmt s_i \left( \frac{1}{\primalvar{i}} \right) \) is convex over the set of positive real numbers  for all \( i = \indexes{\poseig}\). To wit, the objective function in \eqref{eq: primalgeneral} is a finite sum of convex functions, and therefore is convex. The equality constraint is affine and the inequality constraints are convex; therefore, \eqref{eq: primalgeneral} is a convex optimization problem. It should be noted that even though the problem \eqref{eq: primalgeneral} is convex, the objective is sum of inverses, for which, evaluating the gradient is computationally expensive, because of which solving \eqref{eq: primalgeneral} becomes hard due to optimization algorithms generally employing gradient descent schemes. In contrast to the problem \eqref{eq: primalgeneral} let us consider the following optimization problem:
\begin{equation}
\label{eq: dualgeneral}
\begin{aligned}
		& \minimize_{( \dualvar{t} )_{t} \subset \R}	&&  \sum_{t = 1}^{\poseig}  a_t \dualvar{t}^2 - 2  \sqrt{s_t} \dualvar{t}  \\
		& \sbjto						&&
		0 \leq \dualvar{1} \leq \cdots \leq \dualvar{\poseig}.
\end{aligned}
\end{equation}
Since \( a_t > 0 \) for each \( t \), we conclude that \eqref{eq: dualgeneral} is a convex quadratic problem and is easier to solve than \eqref{eq: primalgeneral}.
Moreover we see that, for small enough \( \epsilon > 0 \), the sequence \( (\primalvar{i})_{i = 1}^{\poseig} \) defined by \( \primalvar{1} \Let \big( \sum_{i = 1}^{\poseig} a_i \big) - (\poseig - 1) \epsilon \), and \( \primalvar{i} \Let \epsilon \) for all \( i = 2,\ldots,\poseig, \) satisfies the equality constraint and strictly satisfies the inequality constraints of \eqref{eq: primalgeneral}. Therefore, \( (\eigval{i})_{i = 1}^{\poseig} \) provides the Slater constraint qualification certificate for the problem \eqref{eq: primalgeneral}. Due to convexity and validity of Slater's condition, we conclude that strong duality holds for \eqref{eq: primalgeneral}, which implies that the optimal solution \( (\primalvar{i}\opt )_{i = 1}^{\poseig} \) to \eqref{eq: primalgeneral} and the dual optimal variables must satisfy the KKT conditions \cite[Section 5.5.2, 5.5.3]{boydconvex} (add reference), in this case they are both necessary and sufficient. Therefore, once the dual optimal variables are known, we can compute \( (\primalvar{i}\opt )_{i = 1}^{\poseig} \) by solving the KKT conditions.

We shall show that the problem \eqref{eq: dualgeneral} is indeed equivalent to the Lagrangian dual of \eqref{eq: primalgeneral}, and that the dual optimal variables can be computed easily from the optimal solution to \eqref{eq: dualgeneral}. The following Lemma characterizes the optimal solution to \eqref{eq: primalgeneral} from the optimal solution \( (\dualvar{t}\opt)_{t = 1}^{\poseig} \) of \eqref{eq: dualgeneral}.
\begin{lemma}
\label{lemma: primal dual}
	Consider the optimization problems \eqref{eq: primalgeneral} and \eqref{eq: dualgeneral}. 
	\begin{itemize}[label=\(\circ\), leftmargin=*]
		\item Both \eqref{eq: primalgeneral} and \eqref{eq: dualgeneral} admit unique optimal solutions \( (\primalvar{i}\opt)_{i = 1}^{\poseig} \), \( (\dualvar{t}\opt)_{t = 1} ^{\poseig} \).
		\item Let an ordered subset \( (\partition{1}, \partition{2}, \ldots, \partition{T}) \subset (\indexes{\poseig}) \) be defined iteratively as
		\[%begin{equation}
        \begin{aligned}
        %\label{eq: partition}
        \partition{1} & \Let 1, \\
        \partition{l} & \Let \min \{ t \; \vert \; \partition{(l - 1)} < t \leq \poseig, \; \dualvar{(t - 1)} \opt < \dualvar{t} \opt \} \quad  \text{for all \( l = 2,\ldots, T \).} 
        \end{aligned}
		\]%end{equation}
        \begin{itemize}[label=\(\triangleright\), leftmargin=*]
        \item The unique optimal solution \( (\primalvar{i} \opt)_{i = 1}^{\poseig} \) to                
              \eqref{eq: primalgeneral} is given by
\begin{equation}
( \primalvar{i} \opt )_{i = 1}^{\poseig} = \eigvalmap{(\sqrt{s_i})_{i = 1}^{\poseig}}{(a_i)_{i = 1}^{\poseig}}{(\partition{l})_{l = 1}^{T}}.
\end{equation}  
        \item The optimal values \( p \opt \), \( q \opt \) corresponding to \eqref{eq: primalgeneral} and \eqref{eq: dualgeneral} respectively are given by 
        \begin{equation}
        \label{eq: optimal value primal dual}
          - q \opt = p \opt = \optval{(\sqrt{s_i})_{i = 1}^{\poseig}}{(a_i)_{i = 1}^{\poseig}}{(\partition{l})_{l = 1}^{T}}.
        \end{equation}
        \end{itemize}
\end{itemize}
\end{lemma}

\begin{lemma}
\label{lemma:primal-dual-2}
If the sequence \( (s_i)_{i = 1}^{\poseig} \) in \eqref{eq: primalgeneral} is non-increasing, then,
        \begin{itemize}[label=\( \circ \), leftmargin=*]
			\item the unique optimal sequence \( (\eigval{i}\opt)_{i = 1}^{\poseig} \) is also non-increasing, and
			\item if we have \( \eigval{i}\opt = \eigval{j}\opt = \eigval{} \) for some \( i \neq j \in \{ \indexes{\poseig} \} \), then \( s_i = s_j \).
        \end{itemize} 

\end{lemma}

For \( ( \eigval{i}\opt )_{i =1}^{\poseig} \) to be the unique optimal solution of the problem \eqref{eq: primalgeneral}, let us consider the following optimization problem:
\begin{equation}
\begin{aligned}
\label{eq:opt-perm-map}
\minimize_{\pi \; \in \; \Pi_{\poseig}} \quad \sum_{i = 1}^{\poseig} \frac{s_{\pi(i)}}{\eigval{i}\opt},
\end{aligned}
\end{equation}
where, \( \Pi_{\poseig} \) is the symmetric group on \( (\indexes{\poseig}) \). The optimal permutation map \( \pi\opt \) is characterized by the following lemma.
\begin{lemma}
\label{lemma:primal-dual-3}
Consider the optimization problem \eqref{eq:opt-perm-map}.
\begin{itemize}[label = \( \circ \), leftmargin=*]
\item \eqref{eq:opt-perm-map} admits an optimal solution.

\item If \( \pi\opt \) is an optimal solution to \eqref{eq:opt-perm-map}, then for every \( i \in \{ \indexes{\poseig} \} \), there exist \( j \in \{ \indexes{\poseig} \} \) such that \( \eigval{i}\opt = \eigval{j}\opt \) and \( s_{\pi\opt(i)} = s_j \).

\item A permutation map \( \pi\opt \) is an optimal solution of \eqref{eq:opt-perm-map} \textbf{if and only if} 
\begin{equation}
\label{eq:opt-perm-map-solution}
s_{\pi\opt(1)} \geq s_{\pi\opt(2)} \geq \cdots \geq s_{\pi\opt(\poseig)};
\end{equation}
consequently, \( s_{\pi\opt(i)} = s_i \) for all \( i = \indexes{\poseig} \).
\end{itemize}
\end{lemma}
\begin{remark}
Even though it is straight forward that the condition \eqref{eq:opt-perm-map-solution} is sufficient for the optimality of the permutation map \( \pi\opt \), the fact that it is necessary as well is the crucial part of the assertion of Lemma \ref{lemma:primal-dual-3}. Moreover, it should be noted that the first two assertions of the lemma hold for generic non-increasing sequences \( (\eigval{i})_{i = 1}^{\poseig} \) as well.
\end{remark}

\subsection{Eigenvalues of sum of Hermitian matrices}
\label{subsection: eigenvalues of sum of hermitian matrices}
The following result from \cite[Theorem 2]{wielandt1955extremum} provides a necessary condition on the eigenvalues of sum of hermitian matrices, which we shall use in the proof of Theorem \ref{theorem:DL-theorem-2}.
\begin{theorem}
\label{theorem: eigenvalues of sum of hermitian matrices}
Let \( n \), \( r \) be a positive integers with \( r \leq n \). Let \( A,B,C \in \symmetric{\dimension} \) such that \( C = A + B\), and let the sequences \((\alpha_i)_{i = 1}^{\dimension}\), \((\beta_i)_{i = 1}^{\dimension}\) and \((\gamma)_{i = 1}^{\dimension}\) be the non-increasing sequences of eigenvalues of \(A,B\) and \( C \), respectively. If \( I_r \) is any subset of \( \{ \indexes{\dimension} \} \) with cardinality equal to \( r \), then we have 
\[
\sum_{i \in I_r} \gamma_i \; \leq \; \sum_{i \in I_r} \alpha_i  +  \sum_{i = 1}^r \beta_i.
\]
\end{theorem}

%%%%%%%%%%%%%%%%%%%%%%%%%%%%%%%%%%%%%%%%%%%%%%%%%%%%%%%%%%%%%%%%%%%%%%%%%%%%%%%%
\section{Proofs of main results}
\label{section:proofs-of-main-results}

%===============================================================================

\subsection{Proof of Theorem \ref{theorem:DL-theorem-1}}
For a given dictionary \( \dictionary \in \mathcal{\dictionary}(0) \) of vectors \( ( \dict{i} )_{i = 1}^{\dsize} \) that is feasible for \eqref{eq:DL-problem}, let us define a scheme of representation
\[
	\R^{\dimension} \ni v \mapsto \optscheme_{\dictionary} (v) \Let \pmat{ \dict{1} }{ \dict{2} }{ \dict{\dsize} }^+ v \; \in \R^{\dsize}.
\]
Quite clearly, \( \pmat{ \dict{1} }{ \dict{2} }{ \dict{\dsize} }\optscheme_{\dictionary}(v) = v \) for any \( v \in \R^{\dimension} \) by the definition of the pseudo-inverse because if \( \Span (\dict{i})_{i = 1}^{\dsize} = \R^{\dimension} \), then 
 \( \pmat{ \dict{1} }{ \dict{2} }{ \dict{\dsize} }^+ v \) solves the equation \( \pmat{ \dict{1} }{ \dict{2} }{ \dict{\dsize} } x = v \). Therefore,
\[
	\pmat{ \dict{1} }{ \dict{2} }{ \dict{\dsize} } \optscheme_{\dictionary}(\randomvec) = \randomvec  \text{\ \  \( \mu \)-almost surely}.
\]
We also know that \( \optscheme_{\dictionary}(v) = \pmat{ \dict{1} }{ \dict{2} }{ \dict{\dsize} }^+ v \) is the solution of the least squares problem:
\[
\begin{aligned}
	& \minimize_{x\in\R^{\dsize}}	&& \norm{x}^2 \\
	& \sbjto				&& \pmat{ \dict{1} }{ \dict{2} }{ \dict{\dsize} } x = v.
\end{aligned}
\]                
Therefore, for an arbitrary \( \scheme \in \mathcal{F}(0, \dictionary) \), which also implies that \[
\pmat{ \dict{1} }{ \dict{2} }{ \dict{\dsize} } f(v) = v \quad \text{for all \( v \in \R^{\dimension} \)},
\]
we must have
\begin{align*}
\norm{\optscheme_{\dictionary}(v)}^2  &  \leq \norm{\scheme(v)}^2 \quad\text{for all } v \in \R^{\dimension}.
\end{align*}
Therefore, \( \norm{\optscheme_{\dictionary} (\randomvec)}^2 \leq \norm{\scheme (\randomvec)}^2 \) \( \mathsf{P} \)-almost surely, and hence,
\[
	\EE_{\mathsf{P}} \bigl[ \norm{\optscheme_{\dictionary} (\randomvec)}^2 \bigr]  \leq  \EE_{\mathsf{P}} \bigl[ \norm{\scheme (\randomvec)}^2 \bigr].
\]
Minimizing over all feasible dictionaries and the corresponding schemes, we get
\begin{equation}
	\label{eq:first-comparison}
	\inf_{\dictionary \in \mathcal{\dictionary} (0)}  \EE_{\mathsf{P}} \bigl[ \norm{\optscheme_{\dictionary}(\randomvec)}^2 \bigr] \leq \inf_{\substack{\dictionary \in\mathcal{\dictionary}(0),\\ \scheme \in \mathcal{F} (0, \dictionary)}} \EE_{\mathsf{P}} \bigl[ \norm{\scheme (\randomvec)}^2 \bigr] \\
\end{equation}
The problem on the left-hand side of \eqref{eq:first-comparison} is
\begin{equation}
	\label{eq:first-DL}
	\begin{aligned}
		& \minimize_{( \dict{i} )_{i=1}^{\dsize}} && \EE_{\mathsf{P}} \bigl[ \norm{\optscheme_{\dictionary} (\randomvec)}^2 \bigr] \\
		& \sbjto		&&
		\begin{cases}
			\inprod{ \dict{i} }{\dict{i}} = \length{i} \quad \text{for all \( i = \indexes{\dsize} \),}\\
			\Span ( \dict{i} )_{i = 1}^{\dsize} = \R^{\dimension}. 
		\end{cases}
	\end{aligned}
\end{equation}
From \eqref{eq:first-comparison} we conclude that the optimal value, if it exists, of problem \eqref{eq:DL-problem} is bounded below by the optimal value, if it exists, of \eqref{eq:first-DL}. Our strategy is to demonstrate that \eqref{eq:first-DL} admits a solution, and we shall furnish a feasible solution of \eqref{eq:DL-problem} that achieves a value of the objective function that is equal to the optimal value of the problem \eqref{eq:first-DL}. This will solve \eqref{eq:DL-problem}.

Let \( D \Let \pmat{\dict{1}}{\dict{2}}{\dict{\dsize}} \). The objective function in \eqref{eq:first-DL} is 
\begin{align*}
\EE_{\mathsf{P}} \bigl[ \norm{\optscheme_{\dictionary} (\randomvec)}^2 \bigr] &= \EE_{\mathsf{P}} \bigl[ \norm{D^+ \randomvec}^2 \bigr] \\
&= \EE_{\mathsf{P}} \bigl[ \randomvec\transp (D^+)\transp  D^+ \randomvec \bigr] \\
 &= \EE_{\mathsf{P}} \bigl[   \randomvec\transp \bigl( D\transp (DD\transp )^{-1} \bigr)\transp \bigl( D\transp (DD\transp )^{-1} \bigr) \randomvec \bigr] \\
&= \EE_{\mathsf{P}} \bigl[ \randomvec\transp (DD\transp )^{-1} DD\transp (DD\transp )^{-1} \randomvec \bigr] \\
&= \EE_{\mathsf{P}} \bigl[ \randomvec \transp (DD\transp )^{-1} \randomvec \bigr] \\
&= \EE_{\mathsf{P}} \bigl[ \trace(\randomvec\transp (DD\transp )^{-1} \randomvec) \bigr] \\
&= \EE_{\mathsf{P}} \bigl[ \trace(\randomvec \randomvec\transp (DD\transp )^{-1}) \bigr] \\
&= \trace \left( \EE_{\mathsf{P}} \bigl[ \randomvec \randomvec\transp \bigr] (DD\transp )^{-1} \right).
\end{align*} 
Letting \(  \sigmat{\randomvec}  \Let \EE_{\mathsf{P}} \bigl[ \randomvec \randomvec\transp \bigr] \) and writing \( DD\transp = \sum_{i = 1}^{\dsize} \dict{i}  \dict{i}\transp \), we rephrase \eqref{eq:first-DL} as
\begin{equation}
	\label{eq:DL-1}
	\begin{aligned}
		& \minimize_{( \dict{i} )_{i=1}^{\dsize}}	&&  \trace \biggl( \sigmat{\randomvec}  \biggl( \sum_{i = 1}^{\dsize}  \dict{i}  \dict{i}\transp \biggr)^{-1} \biggr)\\
		& \sbjto		&&
		\begin{cases}
			\inprod{ \dict{i} }{ \dict{i} } = \length{i} \quad \text{for all \( i = \indexes{\dsize} \)} \\
			\Span ( \dict{i} )_{i = 1}^{\dsize} = \R^{\dimension}.
		\end{cases}
	\end{aligned}
\end{equation}
Let \( S \)  be the feasible set for the problem in \eqref{eq:DL-1}. It is clear that \( S \) is non-convex - a family of concentric spheres centered at origin. Let us demonstrate that the objective function of \eqref{eq:DL-1} is convex in the matrix variable \(DD\transp\). We know that whenever \(  \sigmat{\randomvec} \) is a positive definite matrix, invariance of trace with respect to conjugation gives:
\[
	\trace(\sigmat{\randomvec} M^{-1}) = \trace\Bigl(\sigmat{\randomvec}^{1/2} M^{-1} \sigmat{\randomvec}^{1/2} \Bigr) = \trace \Bigl( \bigl( \sigmat{\randomvec}^{-1/2} M \sigmat{\randomvec}^{-1/2} \bigr)^{-1} \Bigr).
\]
From \cite[p.\ 113 and Exercise V.1.15, p.\ 117]{bhatiamatrix} we know that inversion of a matrix is a \emph{matrix convex} map on the set of positive definite matrices. Therefore, for any \( \theta \in [0,1] \) and \( M_1,M_2 \in \posdef{\dimension} \) we have
\begin{multline}
\left( \sigmat{\randomvec}^{-1/2} \bigl( (1-\theta) M_1 + \theta M_2 \bigr) \sigmat{\randomvec}^{-1/2} \right)^{-1}\\
= \left( (1-\theta) \left( \sigmat{\randomvec}^{-1/2} M_1 \sigmat{\randomvec}^{-1/2} \right) + \theta \left( \sigmat{\randomvec}^{-1/2} M_2 \sigmat{\randomvec}^{-1/2} \right) \right)^{-1}\\
\leq (1-\theta) \left( \sigmat{\randomvec}^{-1/2} M_1 \sigmat{\randomvec}^{-1/2} \right)^{-1} + \theta \left( \sigmat{\randomvec}^{-1/2} M_2 \sigmat{\randomvec}^{-1/2} \right)^{-1},
\end{multline}
where \( A \preceq B \) implies that \( B - A \)  is positive semidefinite. Since \( \trace(\cdot) \) is a \emph{linear functional} over the set of \( \dimension \times \dimension \) matrices we have
\begin{multline*}
	\trace \Bigl( \sigmat{\randomvec} \bigl( (1-\theta) M_1 + \theta M_2 \bigr)^{-1} \Bigr) = \trace\Bigl(\Bigl( \sigmat{\randomvec}^{-1/2} \bigl( (1-\theta) M_1 + \theta M_2 \bigr) \sigmat{\randomvec}^{-1/2} \Bigr)^{-1}\Bigr) \\
\leq (1-\theta) \trace \Bigl(\Bigl(\sigmat{\randomvec}^{-1/2} M_1 \sigmat{\randomvec}^{-1/2} \Bigr)^{-1} \Bigr) + \theta \trace\Bigl( \Bigl( \sigmat{\randomvec}^{-1/2} M_2 \sigmat{\randomvec}^{-1/2} \Bigr)^{-1}\Bigr) \\
\leq (1-\theta) \trace ( \sigmat{\randomvec} M_1^{-1}) + \theta \trace ( \sigmat{\randomvec} M_2^{-1} ).
\end{multline*}
In other words, the function \(  M\mapsto \trace ( \sigmat{\randomvec} M^{-1} ) \) is convex on the set of symmetric and positive definite matrices. Moreover, for a collection \( ( \dict{i} )_{i = 1}^{\dsize} \) that is feasible for \eqref{eq:DL-1},
\[
	\mathcal{\dictionary} (0) \ni ( \dict{i} )_{i=1}^{\dsize} \mapsto h( \dict{1} , \ldots,  \dict{\dsize} ) \Let \sum_{i = 1}^{\dsize}  \dict{i}  \dict{i}\transp 
\]
maps into the set of positive definite matrices. Therefore, the objective function in \eqref{eq:DL-1} is convex on \( \image(h) \). This allows us to translate the feasible set of \eqref{eq:DL-1} to the set of matrices \( M \) formed by all feasible collections \( ( \dict{i} )_{i = 1}^{\dsize} \), i.e., on \( h( \mathcal{\dictionary} (0)) \).

For every \( M \in \posdef{\dimension} \), let \( ( \eigval{i}(M) )_{i = 1}^{\dimension} \) denote the sequence of eigenvalues of \( M \) arranged in  non-increasing order. Let us consider a set \( R \) of positive definite matrices whose eigenvalues satisfy the majorization condition of Theorem \ref{theorem: casazza rod}, i.e., \eqref{eq:rod-majorization-condition} is satisfied for the \( \length{i} \)'s of \eqref{eq:DL-1}. In other words, the set \( R \) is defined by 
\[
R \Let \set[\big]{M \in \posdef{\dimension} \suchthat \big( \length{1}, \length{2}, \ldots, \length{\dsize} \big) \majorize \big( \eigval{1}(M), \eigval{2}(M), \ldots, \eigval{\dimension}(M), 0, \ldots, 0 \big) } .
\]
On the one hand, from Theorem \ref{theorem: casazza rod} we know that any positive definite matrix \( M \in R \) can be decomposed as 
\[
	M = \sum_{i = 1}^{\dsize} \dict{i}\dict{i}\transp \quad \text{with } \inprod{ \dict{i} }{\dict{i}} = \length{i} \text{ for all \( i = \indexes{\dsize} \).}
\]
The fact that \( M \) is positive definite implies that \( \Span ( \dict{i} )_{i = 1}^{\dsize} = \R^{\dimension} \). Therefore, \( ( \dict{i} )_{i = 1}^{\dsize} \in \mathcal{\dictionary} (0) \) and \( M = h( \dict{1}, \ldots, \dict{\dsize} ) \), which imply that
\begin{equation}
	\label{4}
	R \subset h(\mathcal{\dictionary} (0)).
\end{equation}
On the other hand, for any dictionary \( \dictionary \) of vectors \( ( \dict{i} )_{i = 1}^{\dsize} \in \mathcal{\dictionary} (0) \), we have \newline \( h( \dictionary ) = \sum_{i = 1}^{\dsize} \dict{i}\dict{i}\transp \in \posdef{\dimension} \), and again from Theorem \ref{theorem: casazza rod} we observe that \( \big( \length{i} \big)_{i = 1}^{\dsize} \majorize \big( \eigval{1} (h( \dictionary )), \ldots, \eigval{\dimension} (h( \dictionary )), 0, \ldots,0 \big) \). Therefore, by definition of \( R \),
\begin{equation}
	\label{5}
	h(\mathcal{\dictionary} (0)) \subset R.
\end{equation}
From \eqref{4} and \eqref{5} we conclude that \( h(\mathcal{\dictionary} (0) ) = R \). The problem \eqref{eq:DL-1} is, therefore, equivalent to:
\begin{equation}
	\label{eq:DL-2}
	\begin{aligned}
		& \minimize_{M \; \in \; \posdef{\dimension} }	&& \trace \bigl(\sigmat{\randomvec} M^{-1}\bigr)\\
		& \sbjto	&& 		\begin{cases} 
		              			 0 < \eigval{\dimension}(M) \leq \cdots \leq \eigval{1}(M) \\
		              			 \big( \length{1}, \length{2}, \ldots, \length{\dsize} \big) \majorize \big( \eigval{1}(M), \ldots, \eigval{\dimension}(M),0, \ldots, 0 \big). 
		                    \end{cases}
	\end{aligned}
\end{equation}

We know that every positive definite matrix \( M \) can be written as \( M = U_M \Lambda_M U_M\transp  \), where \( \Lambda_M \) is a diagonal matrix containing the eigenvalues of \( M \) along the diagonal and \( U_M \) is an orthogonal matrix whose columns are the eigenvectors of \( M \). Alternatively, for any non-increasing sequence \( (\eigval{i})_{i = 1}^{\dimension} \) of positive real numbers that satisfies \( (\length{1}, \ldots, \length{\dsize}) \majorize (\eigval{1}, \ldots, \eigval{\dimension}, 0, \ldots,0) \), let \( \Lambda \Let \diag(\eigval{1}, \ldots, \eigval{\dimension}) \); then we see at once that the positive definite matrix \( M \Let U \Lambda U\transp \) is feasible for \eqref{eq:DL-2} for every orthogonal matrix \( U \). Employing the fact that \( \trace \bigl( \sigmat{\randomvec} U \Lambda^{-1} U\transp \bigr) = \trace \bigl(  U\transp \sigmat{\randomvec} U \Lambda^{-1} \bigr) \), we write the following optimization problem that is equivalent to \eqref{eq:DL-2}.
\begin{equation}
\label{eq:DL-3}
\begin{aligned}
		& \minimize_{\substack{ (\eigval{i})_i \subset \; \R, \\ U \; \in \; \ortho{\dimension} } }	&& \trace \bigl(U\transp \sigmat{\randomvec} U \Lambda^{-1} \bigr)\\
		& \sbjto	&& 		\begin{cases} 
		              			 0 < \eigval{\dimension} \leq \cdots \leq \eigval{1}, \\
		              			 \big( \length{1}, \length{2}, \ldots, \length{\dsize} \big) \majorize \big( \eigval{1}, \ldots, \eigval{\dimension},0, \ldots, 0 \big), \\
		              			 \Lambda = \diag(\eigval{1}, \ldots, \eigval{\dimension}). 
		                    \end{cases}
	\end{aligned}
\end{equation}

For every non-increasing sequence \( (\eigval{i})_{i = 1}^{\dimension} \) that is feasible for \eqref{eq:DL-3}, let us consider the following optimization problem:
\begin{equation}
\label{eq:optX}
		 \minimize_{U \; \in \; \ortho{\dimension} }	\quad \trace \bigl( U\transp \sigmat{\randomvec} U \Lambda^{-1} \bigr).
\end{equation}
Let \( ( \eigvalsig{i} )_{i = 1}^{\dimension} \) be the non-increasing sequence of eigenvalues of \( \sigmat{\randomvec} \) and \( ( \eigvalsig{i}' )_{i = 1}^{\dimension} \) be the diagonal entries of the matrix \( U\transp \sigmat{\randomvec} U \) for \( U \in \ortho{\dimension} \). Let \( \pi \) be a permutation map on \( (\indexes{\dimension}) \) such that \(  \eigvalsig{\pi(1)}' \geq  \eigvalsig{\pi(2)}' \geq \cdots  \eigvalsig{\pi(\dimension)}' \). Since, the sequence \( (\eigval{i})_{i = 1}^{\dimension} \) is in non-increasing order, we know that the sequence \( (\frac{1}{\eigval{i}})_{i = 1}^{\dimension} \) is non-decreasing, and from the rearrangement inequality \eqref{rearrangement inequality}, we have, 
\begin{equation}
\sum_{i = 1}^{\dimension} \frac{\eigvalsig{\pi(i)}'}{\eigval{i}} \; \leq \; \sum_{i = 1}^{\dimension} \frac{\eigvalsig{i}'}{\eigval{i}}.
\end{equation}
Therefore, the optimization problem \eqref{eq:optX} reduces to
\begin{equation}
\label{eq:optX1}		 
	\begin{aligned}
		& \minimize_{ U \; \in \; \ortho{\dimension} }  &&  \sum_{i = 1}^{\dimension} \frac{\eigvalsig{i}'}{\eigval{i}} \\
		& \sbjto	&& 	\begin{cases}	
		                   \eigvalsig{i}' = \inprod{e_i}{U\transp \sigmat{\randomvec} U e_i} \quad \text{for all \( i = \indexes{\dimension} \),} \\
		                   \eigvalsig{\dimension}' \leq \cdots \leq \eigvalsig{1}'.	
		                \end{cases}              			 
	\end{aligned}
\end{equation}
Let \(U \in \ortho{\dimension} \) be feasible for \eqref{eq:optX1} and \( (\eigvalsig{i}')_{i = 1}^{\dimension} \) be the corresponding non-increasing sequence of the diagonal entries of the matrix \( U\transp \sigmat{\randomvec} U \) whose eigenvalues are \( (\eigvalsig{i})_{i = 1}^{\dimension} \). From the \emph{Schur-Horn} Theorem \ref{theorem: schur horn} it follows that \( (\eigvalsig{i}')_{i = 1}^{\dimension} \majorize (\eigvalsig{i})_{i = 1}^{\dimension} \), and therefore, from  \emph{Kadison's} Lemma \ref{lemma: Kadison} we conclude that \( \R^{\dimension} \ni \eigvalsig{}' \Let (\eigvalsig{1}', \ldots, \eigvalsig{\dimension}')\transp \in \permpolytope{(\eigvalsig{1}, \ldots, \eigvalsig{\dimension})\transp} \). Recall that for any \( y \in \R^{\dimension} \), \( \permpolytope{y} \) is the permutation polytope of \( y \), it is a bounded polytope in \( \R^{\dimension} \) whose extreme points are all the permutations of \( y \). By defining \( \eigval{}^{-1} \), \( \eigvalsig{} \in \R^{\dimension} \) by \( \eigval{}^{-1} \Let (\frac{1}{\eigval{1}}, \ldots, \frac{1}{\eigval{\dimension}})\transp \) and \( \eigvalsig{} \Let (\eigvalsig{1}, \ldots, \eigvalsig{\dimension})\transp \), it is immediate that the optimum value in \eqref{eq:optX1} is bounded below by the optimum value of:
\begin{equation}
\label{eq:optX2}
\minimize_{ x \; \in \; \permpolytope{\eigvalsig{}} }	\quad \inprod{x}{\eigval{}^{-1}}.
\end{equation}
The optimization problem \eqref{eq:optX2} is a linear program, and from the fundamental theorem of linear programming we know that one of the extreme points of \( \permpolytope{\eigvalsig{}} \) is an optimal solution. From the fact that the extreme points of \( \permpolytope{\eigvalsig{}} \) are the vectors obtained by permuting the components of \( (\eigvalsig{1}, \ldots, \eigvalsig{\dimension})\transp \), \eqref{eq:optX2} reduces to:
\begin{equation}
\label{eq:optX3}
\minimize_{ \pi \; \in \; \Pi_{\dimension} }  \quad  \inprod{\eigvalsig{}(\pi)}{\eigval{}^{-1}},
\end{equation}
where \( \Pi_{\dimension} \) is the symmetric group on \( ( \indexes{\dimension} ) \) and \( \eigvalsig{}(\pi) \Let (\eigvalsig{\pi(1)}, \ldots, \eigvalsig{\pi(\dimension)})\transp \). Since the sequences \( ( \frac{1}{\eigval{i}} )_{i = 1}^{\dimension} \) and \( ( \eigvalsig{i} )_{i = 1}^{\dimension} \) are non-decreasing and non-increasing respectively, the rearrangement inequality \eqref{rearrangement inequality} implies that,
\[
\sum\limits_{i = 1}^{\dimension} \frac{\eigvalsig{i}}{\eigval{i}} \leq \sum\limits_{i = 1}^{\dimension} \frac{\eigvalsig{\pi(i)}}{\eigval{i}}. 
\]

We note that no characterization of an optimal solution to \eqref{eq:optX} has been given so far. We shall revisit \eqref{eq:optX} with \( ( \eigval{i} )_{i = 1}^{\dimension} = ( \eigval{i}\opt )_{i = 1}^{\dimension} \) (the optimal sequence), and characterize an optimal solution \( X\opt \) for this special case, which is sufficient. 
 
Thus, the optimization problem \eqref{eq:DL-3} reduces to the following:
\begin{equation}
\label{eq:DL-5}
\begin{aligned}
		& \minimize_{( \eigval{i} )_i \subset \; \R }	&& \sum_{i = 1}^{\dimension} \frac{\eigvalsig{i}}{\eigval{i}} \\
		& \sbjto	&& 		\begin{cases} 
		              			 0 < \eigval{\dimension} \leq \cdots \leq \eigval{1}, \\
		              			 ( \length{1}, \length{2}, \ldots, \length{\dsize} ) \majorize ( \eigval{1}, \ldots, \eigval{\dimension},0, \ldots, 0 ). 
		                    \end{cases}
\end{aligned}
\end{equation}
Let us define \( (\length{i}')_{i = 1}^{\dimension} \) by \( \length{i}' \Let \length{i} \) for all \( i = \indexes{\dimension - 1} \) and \( \length{\dimension}' \Let \sum_{i = \dimension}^{\dsize} \length{i} \). By eliminating the constraint \( \eigval{\dimension} \leq \cdots \leq \eigval{1} \) in \eqref{eq:DL-5} and rewriting the constraint \( ( \length{1}, \length{2}, \ldots, \length{\dsize} ) \majorize ( \eigval{1}, \ldots, \eigval{\dimension},0, \ldots, 0 ) \) in terms of the sequence \( ( \length{i}' )_{i = 1}^{\dimension} \), we arrive at:
\begin{equation}
\label{eq:DL-6}
\begin{aligned}
		& \minimize_{ ( \eigval{i} )_i \subset \; \R }	&& \sum_{i = 1}^{\dimension} \frac{\eigvalsig{i}}{\eigval{i}} \\
		& \sbjto	&& 		\begin{cases} 
		              			  \sum_{i = 1}^j \length{i}' \; \leq \; \sum_{i = 1}^j \eigval{i} \quad \text{for all \( j = \indexes{\dimension - 1} \),} \\
\sum_{i = 1}^{\dimension} \length{i}' \; = \; \sum_{i = 1}^{\dimension} \eigval{i} .
		                    \end{cases}
	\end{aligned}
\end{equation}
It is obvious that every sequence \( ( \eigval{i} )_{i = 1}^{\dimension} \) that is feasible for \eqref{eq:DL-5} is also feasible for \eqref{eq:DL-6}, hence the optimal value of \eqref{eq:DL-6} is a lower bound for that of \eqref{eq:DL-5}. The problem \eqref{eq:DL-5} is a variant of \eqref{eq: primalgeneral}, and therefore from Lemma \ref{lemma: primal dual} we conclude that \eqref{eq:DL-6} admits a unique optimal solution \( (\eigval{i}\opt)_{i = 1}^{\dimension} \). Since \( (\eigvalsig{i})_{i = 1}^{\dimension} \) is a non-increasing sequence, we also conclude that the optimal solution \( ( \eigval{i}\opt )_{i = 1}^{\dimension} \) satisfies \( \eigval{\dimension}\opt \leq \cdots \leq \eigval{1}\opt \). Therefore, \( (\eigval{i}\opt)_{i = 1}^{\dimension} \) is also feasible for \eqref{eq:DL-5}. This implies that \( (\eigval{i}\opt)_{i = 1}^{\dimension} \) is the unique optimal solution to the optimization problem \eqref{eq:DL-5}.

From Lemma \ref{lemma: primal dual} we know that the optimization problem 
\begin{equation}
\begin{aligned}
		& \minimize_{( \dualvar{t} )_t \subset \; \R}	&&  \sum_{t = 1}^{\dimension}  \length{t}' \dualvar{t}^2 - 2  \sqrt{\eigvalsig{t}} \dualvar{t}  \\
		& \sbjto						&&
		0 \leq \dualvar{1} \leq \cdots \leq \dualvar{\dimension},
\end{aligned}
\end{equation}
admits a unique optimal solution \( (\dualvar{t} \opt)_{t = 1}^{\dimension} \) with an optimal value of \( q \opt \). If an ordered set \( (\partition{1}, \partition{2}, \ldots, \partition{T}) \subset (\indexes{\dimension}) \) is defined as
		\[%begin{equation}
        \begin{aligned}
        \label{eq: partition}
        \partition{1} & \Let 1, \\
        \partition{l} & \Let \min \{ t \; \vert \; \partition{(l - 1)} < t \leq \dimension, \; \dualvar{(t - 1)} \opt < \dualvar{t} \opt \} \quad  \text{for all \( l = 2,\ldots, T \)}, 
        \end{aligned}
		\]%end{equation}
then the optimal solution \( (\eigval{i} \opt)_{i = 1}^{\dimension} \) and the optimal value \( p\opt \) of \eqref{eq:DL-5} are given by
\begin{equation}
\label{eq:optimal-solution-to-DL5}
\begin{aligned}
(\eigval{i} \opt)_{i = 1}^{\dimension} & = \eigvalmap{(\sqrt{\eigvalsig{i}})_{i = 1}^{\dimension}}{(\length{i}')_{i = 1}^{\dimension}}{(\partition{l})_{l = 1}^{T}}, \\
- q \opt = p \opt & = \optval{(\sqrt{s_i})_{i = 1}^{\poseig}}{(a_i)_{i = 1}^{\poseig}}{(\partition{l})_{l = 1}^{T}}.
\end{aligned}
\end{equation}

We note that we have given a characterization only for the optimal \( ( \eigval{i}\opt )_{i = 1}^{\dimension} \) in \eqref{eq:DL-3} and not for the optimal orthogonal matrix \( U\opt \) there. We need both for a complete characterization of optimal solution \( M\opt \) of \eqref{eq:DL-2}. To that end, let us consider the following instance of \eqref{eq:optX}
\begin{equation}
\label{eq:optimalU}
		 \minimize_{U \; \in \; \ortho{\dimension} }	\quad \trace \bigl( U\transp \sigmat{\randomvec} U {\Lambda\opt}^{-1} \bigr),
\end{equation}
where \( \Lambda\opt \Let \diag (\eigval{1}\opt, \ldots, \eigval{\dimension}\opt ) \).
\begin{claim}
\label{claim:optimalU}
An orthogonal matrix \( U\opt \Let \pmat{u\opt_1}{u\opt_2}{u\opt_{\dimension}} \) is an optimal solution to \eqref{eq:optimalU} \textbf{if and only if}  \( \sigmat{\randomvec}u\opt_i = \eigvalsig{i} u\opt_i \) for all \( i = 1, 2, \ldots, \dimension \).
\end{claim}
\begin{proof}
As seen earlier in the proof, \eqref{eq:optimalU} reduces to:
\begin{equation}
\label{eq:opt-perm-in-theorem}
\minimize_{\pi \in \Pi_{\dimension} } \quad \sum_{i = 1}^{\dimension} \frac{\eigvalsig{\pi(i)}}{\eigval{i}\opt}.
\end{equation}
Since the sequence \( (\eigval{i}\opt)_{i = 1}^{\dimension} \) is the optimal solution to \eqref{eq:DL-6}, in view of Lemma \ref{lemma:primal-dual-3} we conclude that a permutation map \( \pi\opt \) is an optimal solution of \eqref{eq:opt-perm-in-theorem} if and only if \( \eigvalsig{\pi\opt(1)} \geq \eigvalsig{\pi\opt(2)} \geq \cdots \geq \eigvalsig{\pi\opt(\dimension)} \) and consequently \( \eigvalsig{\pi\opt(i)} = \eigvalsig{i} \) for all \( i = \indexes{\dimension} \), it should be noted that there are permutation maps other than the identity that also are optimal solutions to \eqref{eq:opt-perm-in-theorem}. For \( U\opt \) to be an optimal solution of \eqref{eq:opt-perm-in-theorem}, we need the diagonal entries of the matrix \( {U\opt}\transp \sigmat{\randomvec} U\opt \) to be equal to and arranged in the order of \( (\eigvalsig{\pi\opt(i)})_{i = 1}^{\dimension} \), and this happens if and only if \( \sigmat{\randomvec} u\opt_i = \eigvalsig{\pi\opt(i)} u\opt_i = \eigvalsig{i} u\opt_i \) for all \( i = \indexes{\dimension} \).  
\end{proof}

Returning back to \eqref{eq:DL-3}, we observe that even though the optimal sequence \( (\eigval{i})_{i = 1}^{\dimension} \) is unique, an optimal orthogonal matrix \( U\opt \) is, however, not. It is now obvious that every optimal solution \( M\opt \) to \eqref{eq:DL-2} is of the form:
\[
M\opt \Let \sum\limits_{i = 1}^{\dimension} \eigval{i}\opt \; u\opt_i {u\opt_i}\transp,
\]
where \( (u\opt_i)_{i = 1}^{\dimension} \) is \emph{any} orthonormal sequence of eigenvectors of \( \sigmat{\randomvec} \) corresponding to the eigenvalues \( (\eigvalsig{i})_{i = 1}^{\dimension} \).

From feasibilty of \( (\eigval{i}\opt)_{i = 1}^{\dimension} \) for \eqref{eq:DL-3}, we conclude from Theorem \ref{theorem: casazza rod} that there exists a sequence of vectors \( (\dict{i}\opt)_{i = 1}^{\dsize} \subset \R^{\dimension} \) such that \( \inprod{\dict{i}\opt}{\dict{i}\opt} = \length{i} \) for all \( i = \indexes{\dsize} \), and that \( M\opt \) admits the rank-\( 1 \) decomposition
\begin{equation}
\label{eq:Mstar-decomposition}
M\opt = \sum_{i = 1}^{\dsize} \dict{i}\opt { \dict{i}\opt }\transp.
\end{equation}
Clearly, the dictionary \( \dictionary\opt \Let (\dict{i}\opt)_{i = 1}^{\dsize} \) is feasible for \eqref{eq:DL-1} and from the equivalence of optimization problems \eqref{eq:DL-1} and \eqref{eq:DL-2} we conclude that the dictionary \( \dictionary\opt \) is optimal for \eqref{eq:DL-1}.

It remains to define the optimal scheme. Let us define
\[
	\R^{\dimension} \ni v \mapsto \scheme\opt (v) \Let \pmat{\dict{1}\opt}{\dict{2}\opt}{\dict{\dsize}\opt}^+ v \in \R^{\dsize}.
\]
It is evident that \( \scheme\opt \) is feasible for \eqref{eq:DL-problem}. But then the objective function in \eqref{eq:DL-problem} evaluated at \( \dictionary = \dictionary\opt \) and \( \scheme = \optscheme \) must be equal to \( p\opt \). Since \( p\opt \) is also a lower bound for the optimal value of \eqref{eq:DL-problem}, the problem \eqref{eq:DL-problem} admits a solution. An optimal dictionary-scheme pair is, therefore, given by
\begin{equation}
	\label{DL solution 1}
	\begin{cases}
		\dictionary\opt = ( \dict{i}\opt )_{i = 1}^{\dsize} \text{ obtained from the decomposition \eqref{eq:Mstar-decomposition}, and} \\
		\R^{\dimension} \ni v\mapsto \scheme\opt(v) = \pmat{\dict{1}\opt}{\dict{2}\opt}{\dict{\dsize}\opt}^+ v \in \R^{\dsize} .
	\end{cases}
\end{equation}
The proof is now complete.

%===============================================================================
\subsection{Proof of Lemma \ref{lemma 2}}
\label{subsection:proof-of-lemma2}

We argue by contradiction, and suppose that the assertion of the Lemma is false. If we denote by $x_i$ the orthogonal projection of $\dict{i}$ on \( X_{\randomvec} (\const) \) and by \( y_i \) the orthogonal projection of \( \dict{i} \) on the orthogonal complement of \( X_{\randomvec} (\const) \), we must have \( \norm{x_i} < \sqrt{\length{i}} \) for at least one value of \( i \). If \( \scheme \) is an optimal scheme of representation, feasibility of \( \scheme \) gives, for any \( v \in R_{\randomvec} (\const) \),
\begin{equation}\label{lemma 2 eq 1}
\begin{aligned}
	v - \const	& = \sum_{i = 1}^K \dict{i} f_i(v) = \biggl(\sum_{i = 1}^K x_i f_i(v)\biggr) + \biggl(\sum_{i = 1}^K y_i f_i(v)\biggr)\\
		 & = \sum_{\substack{i = 1,\\ \norm{x_i} \neq 0}}^K x_i f_i(v) + 0.
\end{aligned}
\end{equation}
Fix a unit vector \( x\in X_{\randomvec} (\const) \), and define a dictionary \( ( \dict{i}\opt )_{i=1}^{\dsize} \) by
\begin{align*}
\dict{i}\opt \Let 
\begin{dcases}
\frac{\sqrt{\length{i}}}{\norm{x_i}} x_i & \text{if } \norm{x_i} \neq 0,\\
\sqrt{\length{i}} x	& \text{otherwise}.
\end{dcases}
\end{align*}
We see immediately that 
\[
	\Span ( \dict{i}\opt )_{i = 1}^{\dsize} \supset \Span ( x_i )_{i = 1}^{\dsize} \supset R_{\randomvec} (\const) \quad \text{and} \quad \inprod{\dict{i}\opt}{\dict{i}\opt} = \length{i} \text{ for all }i = \indexes{\dsize}.
\]
In other words, the dictionary of vectors \( ( \dict{i}\opt )_{i = 1}^{\dsize} \) is feasible for the problem \eqref{eq:DL-problem-general}. Let us now define a scheme \( \optscheme \) by
\[
\R^{\dimension} \ni v\mapsto \optscheme (v) \Let \diag \left( \frac{\norm{x_1}}{\sqrt{\length{1}}}, \frac{\norm{x_2}}{\sqrt{\length{2}}}, \ldots,\frac{\norm{x_{\dsize}}}{\sqrt{\length{{\dsize}}}} \right) \scheme (v) \in \R^{\dsize}.
\] 
For any \( v \in R_{\randomvec} (\const) \), using the dictionary of vectors \( ( \dict{i}\opt )_{i = 1}^{\dsize} \), we get
\begin{equation}\label{lemma 2 eq 2}
\begin{aligned}
	\sum_{i = 1}^{\dsize} \dict{i}\opt \optscheme_i (v) &= \sum_{i = 1}^{\dsize} \dict{i}\opt \frac{\norm{x_i}}{\sqrt{\length{i}}} \scheme_i (v) = \sum_{\substack{i = 1,\\\norm{x_i} \neq 0}}^{\dsize} \frac{\norm{x_i}}{\sqrt{\length{i}}} \frac{\sqrt{\length{i}}}{\norm{x_i}} x_i \scheme_i(v) = v,
\end{aligned}
\end{equation}
where the last equality follows from \eqref{lemma 2 eq 1}. Thus, \(  \optscheme (\cdot) \) along with the dictionary of vectors \( ( \dict{i}\opt )_{i = 1}^{\dsize} \) is feasible for problem \eqref{eq:DL-problem-general}. But for any \( v \in R_{\randomvec} (\const) \) we have
\[
\norm{\optscheme (v)}^2 = \sum_{i = 1}^{\dsize} \bigl(\optscheme_i (v) \bigr)^2 = \sum_{i = 1}^{\dsize} \frac{\norm{x_i}^2}{\length{i}} \bigl(\scheme_i (v) \bigr)^2 < \sum_{i = 1}^{\dsize} \bigl( \scheme_i (v) \bigr)^2 = \norm{\scheme (v) }^2,
\]
where the strict inequality holds due to the fact that \( \norm{x_i} < \sqrt{\length{i}} \) for at least one \( i \). However, this contradicts the assumption that the pair \( ( \dict{i} )_{i = 1}^{\dsize} \) along with the scheme \( \scheme \) is optimal for \eqref{eq:DL-problem-general}, and the assertion follows.

%===============================================================================
\subsection{Proof of Theorem \ref{theorem:DL-theorem-2}}

Consider \eqref{eq:DL-problem-general}, fix some \( \const \in \R^{\dimension} \), and let the dimension of \( X_{\randomvec} (\const) \) be \( \dimension (\const) \) with \( 
\dimension (\const) < \dimension \). It should be noted that \( X_{\randomvec} (\const) = \image(\sigmat{\randomvec - \const} ) \), and therefore, a basis of \( X_{\randomvec} (\const) \) can be obtained by computing the eigenvectors of \( \sigmat{\randomvec - \const} \) corresponding to positive (non-zero) eigenvalues. Let \( ( u_i (\const) )_{i = 1}^{\dimension (\const)} \) be the orthonormal eigenvectors of \( \sigmat{\randomvec - \const} \) corresponding to the eigenvalues \( ( \eigvalsig{i} (\const) )_{i = 1}^{\dimension (\const)} \). Let us define 
\begin{align*}
   U(\const) &\Let \pmat{u_1 (\const)}{u_2 (\const)}{u_{\dimension (\const)} (\const)}, \\
   \projectedrv &\Let {U (\const)}\transp (\randomvec - \const).
\end{align*}

Since \( ( u_i (\const) )_{i = 1}^{\dimension (\const)} \) is a basis for \( X_{\randomvec} (\const) \), every vector in \(  X_{\randomvec} (\const) \) can be uniquely represented using this basis. We know that \( \randomvec - \const \) takes values in \(  X_{\randomvec} (\const) \) \( \PP \)-almost surely. Clearly \( \projectedrv \) is the unique representation of \( \randomvec - \const \) in terms of the basis \( ( u_i (\const) )_{i = 1}^{\dimension (\const)} \). Similarly, let $\dictx{i}$ be the representation of the dictionary vector $\dict{i}$ in the basis $ ( u_i (\const) )_{i = 1}^{\dimension (\const)} $, i.e., $\dict{i} = U (\const) \dictx{i}$. The constraints on the dictionary vectors can now be equivalently written as :
\begin{itemize}[label=\(\circ\), leftmargin=*]
\item \( \inprod{\dict{i}}{\dict{i}} = \length{i} \quad \Rightarrow \quad \inprod{\dictx{i}}{\dictx{i}} = \length{i} \), and
\item \( \Span ( \dict{i} )_{i = 1}^{\dsize} \supset R_{\randomvec} (\const) \quad \Rightarrow \quad \Span ( \dict{i} )_{i = 1}^{\dsize} = X_{\randomvec} (\const) \quad\Rightarrow\quad \Span ( \dictx{i} )_{i = 1}^{\dsize} = \R^{\dimension (\const)} \).
\end{itemize}

For every feasible scheme \( f \) let us define an associated scheme for representing samples of the random vector \( \projectedrv \) by
\[
	\R^{\dimension (\const)} \ni v \longmapsto \scheme_{\const} (v) \Let \scheme \big( U(\const)v + \const )\in\R^K.
\]
The conditions on feasibility of $f$ in \eqref{eq:DL-problem-general} imply that the scheme $\scheme_{\const}$ is feasible whenever for some feasible dictionary of vectors $( \dictx{i} )_{i = 1}^{\dsize}$ we have
\[
	\pmat{\dictx{1}}{\dictx{2}}{\dictx{{\dsize}}} \scheme_{\const}(\projectedrv) = \projectedrv \quad\text{\( \PP \)-almost surely.}
\]
In other words, in contrast to the problem \eqref{eq:DL-problem-general}, where the optimization is carried out over vectors in $\mathbb{R}^{\dimension}$, we can equivalently consider the same problem in $\mathbb{R}^{\dimension (\const)}$ but with the following modified constraints: 
\begin{equation}\label{eq:DL-problem-general-2}
\begin{aligned}
	& \minimize_{\const, \; ( \dictx{i} )_i, \; \scheme_{\const}}	&& \EE_{\PP} \bigl[ \inprod{\scheme_{\const} (\projectedrv)}{\scheme_{\const} (\projectedrv)} \bigr]\\
	& \sbjto									&&
	\begin{cases}
		\inprod{\dictx{i}}{\dictx{i}} = \length{i} \text{ for all $i = 1,\ldots,{\dsize},$}\\
		\Span ( \dictx{i} )_{i = 1}^{\dsize} = \mathbb{R}^{\dimension (\const)},\\
		\pmat{\dictx{1}}{\dictx{2}}{\dictx{{\dsize}}} \scheme_{\const}(\projectedrv) = \projectedrv \text{\ \ \(\PP\)-almost surely}.
	\end{cases}
\end{aligned}
\end{equation}

For a fixed value of \( \const \), the problem \eqref{eq:DL-problem-general-2} is a version of \eqref{eq:DL-problem}, and the solutions of the latter are characterized by Theorem \ref{theorem:DL-theorem-1}. To prove the assertions of Theorem \ref{theorem:DL-theorem-2} it is sufficient to prove that \( \const\opt \), the optimum value of \( \const \) in \eqref{eq:DL-problem-general-2}, is equal to \( \mu = \EE_{\PP} [\randomvec] \).

Denoting by \( (\eigvalsig{i} (\const))_{i = 1}^{\dimension (\const)} \) the eigenvalues of \( \sigmat{\projectedrv} \), we conclude from the proof of Theorem \ref{theorem:DL-theorem-1} that the optimum value of the problem \eqref{eq:DL-problem-general-2} is equal to that of:
\begin{equation}
\label{eq:DL-problem-primal-dual-theorem-2}
\begin{aligned}
		& \minimize_{ ( \eigval{i} )_i }	&& \sum_{i = 1}^{\dimension (\const)} \frac{\eigvalsig{i} (\const)}{\eigval{i}} \\
		& \sbjto	&& 		\begin{cases} 
		              			 0 < \eigval{\dimension (\const)} \leq \cdots \leq \eigval{1}, \\
		              			 ( \length{1}, \length{2}, \ldots, \length{\dsize} ) \majorize ( \eigval{1}, \ldots, \eigval{\dimension (\const)},0, \ldots, 0 ). 
		                    \end{cases}
\end{aligned}
\end{equation}

Let \( \mu \Let \EE_{\PP} [\randomvec] \), \( \sigmat{}\opt \Let \var (\randomvec) \) and let \( (\eigvalsig{i}\opt)_{i = 1}^{\poseig} \) be the sequence of positive eigenvalues of \( \sigmat{}\opt \) arranged in non-increasing order. Let us consider the following version of \eqref{eq:DL-problem-primal-dual-theorem-2}:
\begin{equation}
\label{eq:DL-problem-primal-dual-1-theorem-2}
\begin{aligned}
		& \minimize_{ ( \eigval{i} )_i }	&& \sum_{i = 1}^{\poseig} \frac{\eigvalsig{i}\opt }{\eigval{i}} \\
		& \sbjto	&& 		\begin{cases} 
		              			 0 < \eigval{\poseig} \leq \cdots \leq \eigval{1} \\
		              			 ( \length{1}, \length{2}, \ldots, \length{\dsize} ) \majorize ( \eigval{1}, \ldots, \eigval{\poseig},0, \ldots, 0 ). 
		                    \end{cases}
\end{aligned}
\end{equation}
We know from Lemma \ref{lemma: primal dual} that both \eqref{eq:DL-problem-primal-dual-theorem-2} and \eqref{eq:DL-problem-primal-dual-1-theorem-2} admit unique optimal solutions, say \( (\eigval{i}\opt (\const))_{i = 1}^{\dimension (\const)} \) and \( (\eigval{i}\opt)_{i = 1}^{\poseig} \), respectively. We shall prove that the optimum value of \eqref{eq:DL-problem-primal-dual-theorem-2} is bounded below by that of \eqref{eq:DL-problem-primal-dual-1-theorem-2}.

Observe that \( \sigmat{\projectedrv} = {U(\const)}\transp \big( \sigmat{\randomvec - \const} \big) U(\const) \). Therefore, the eigenvalues of \( \sigmat{\projectedrv} \) are the positive eigenvalues of \( \sigmat{\randomvec - \const} \). By definition 
\begin{align*}
\sigmat{}\opt &= \EE_{\PP}[ (\randomvec - \mu) (\randomvec - \mu)\transp ] \\ 
& = \EE_{\PP}[ (\randomvec - \const + \const - \mu ) (\randomvec - \const + \const - \mu )\transp ] \\
& = \sigmat{\randomvec - \const} - (\mu - \const) (\mu - \const)\transp.
\end{align*}
Since \( (\mu - \const) (\mu - \const)\transp \in \possemdef{\dimension} \), we conclude that the maximal eigenvalue of \( -(\mu - \const) (\mu - \const)\transp \) is equal to \( 0 \). Let \( r \Let 1 \), and let \( (\eigvalsig{i}\opt)_{i = 1}^{\dimension} \) and \( (\eigvalsig{i} (\const))_{i = 1}^{\dimension} \) be the non-increasing sequence of eigenvalues of \( \sigmat{}\opt \) and \( \sigmat{\randomvec - \const} \), respectively. From Theorem \ref{theorem: eigenvalues of sum of hermitian matrices} we conclude that \( \eigvalsig{i}\opt \leq \eigvalsig{i}(\const) \) for all \( i = \indexes{\dimension} \), which further implies that \( \poseig \leq \dimension(\const) \) for every \( \const \in \R^{\dimension} \). Also, it is readily verified that \( \poseig = \dimension (\const) \) and \( \eigvalsig{i}\opt = \eigvalsig{i}(\const) \) for all \( i = \indexes{\poseig} \), if and only if \( \const = \mu \).

We define \( \beta \geq 1 \) such that \( \beta \sum\limits_{i = 1}^{\poseig} \eigval{i}\opt(\const) = \sum\limits_{i = 1}^{\dimension (\const)} \eigval{i}\opt(\const) \), and a new sequence \( (\eigval{i}')_{i = 1}^{\poseig} \) by \( \eigval{i}' \Let \beta \eigval{i}\opt (\const) \) for all \( i = \indexes{\poseig} \). We see immediately that 
\begin{align*}
\sum_{i = 1}^{j} \eigval{i}' &= \beta \sum_{i = 1}^{j} \eigval{i}\opt (\const) \geq \sum_{i = 1}^{j} \eigval{i}\opt (\const) \geq \sum_{i = 1}^j \length{i} \quad \text{for all \( j = \indexes{\poseig - 1} \)}, \\
\sum_{i = 1}^{\poseig} \eigval{i}' &= \beta \sum_{i = 1}^{\poseig} \eigval{i}\opt (\const) = \sum_{i = 1}^{\dimension (\const)} \eigval{i}\opt (\const) = \sum_{i = 1}^{\dsize} \length{i}.
\end{align*}
In other words, \( (\eigval{i}')_{i = 1}^{\poseig} \) is feasible for the optimization problem \eqref{eq:DL-problem-primal-dual-1-theorem-2}. We also note that 
\begin{align*}
\sum_{i = 1}^{\poseig} \frac{\eigvalsig{i}\opt}{\eigval{i}\opt} &\leq \sum_{i = 1}^{\poseig} \frac{\eigvalsig{i}\opt}{\eigval{i}'} = \frac{1}{\beta} \sum_{i = 1}^{\poseig} \frac{\eigvalsig{i}\opt}{\eigval{i}\opt(\const)} \leq \sum_{i = 1}^{\poseig} \frac{\eigvalsig{i}\opt}{\eigval{i}\opt(\const)} \leq \sum_{i = 1}^{\poseig} \frac{\eigvalsig{i} (\const)}{\eigval{i}\opt(\const)} \leq \sum_{i = 1}^{\dimension (\const)} \frac{\eigvalsig{i} (\const)}{\eigval{i}\opt(\const)},
\end{align*}
and equality holds in the preceding chain if and only if \( \poseig = \dimension (\const) \) and \( \eigvalsig{i}\opt = \eigvalsig{i}(\const) \) for all \( i = \indexes{\poseig} \); equivalently, \emph{if and only if} \( \const = \mu \). Therefore, \( \const\opt = \mu \) is the unique optimal solution to \eqref{eq:DL-problem-primal-dual-1-theorem-2}. 

Let \( (\dictx{i})_{i = 1}^{\dsize} \) be an optimal solution to \eqref{eq:DL-problem-primal-dual-1-theorem-2} with \( \const = \const\opt \). From Theorem \ref{theorem:DL-theorem-1} we know that 
\begin{equation}
\label{eq:rank-1-opt-dictionary}
\sum_{i = 1}^{\dsize} \dictx{i}\opt {\dictx{i}\opt}\transp = \sum_{i = 1}^{\poseig} \eigval{i}\opt u_i' {u_i'}\transp
\end{equation}
for some orthonormal sequence of eigenvectors of \( \sigmat{\projectedrv} \) corresponding to the eigenvalues \( (\eigvalsig{i}\opt)_{i = 1}^{\poseig} \). Since \( \sigmat{}\opt U(\const\opt) u_i' = \eigvalsig{i}\opt U(\const\opt) u_i' \) for all \( i = \indexes{\poseig} \), the optimal dictionary \( (\dict{i}\opt)_{i = 1}^{\dsize} \Let \big( U(\const\opt) \dictx{i}\opt \big)_{i = 1}^{\dsize} \) satisfies: 
\begin{equation}
\label{eq:rank-1-opt-dictionary-1}
\sum_{i = 1}^{\dsize} \dict{i}\opt {\dict{i}\opt}\transp = \sum_{i = 1}^{\poseig} \eigval{i}\opt u_i {u_i}\transp,
\end{equation}
where \( (u_i)_{i = 1}^{\poseig} \Let \big( U(\const\opt)u_i' \big)_{i = 1}^{\poseig} \) is the sequence of orthonormal eigenvectors of \( \sigmat{}\opt \) corresponding to the eigenvalues \( (\eigvalsig{i}\opt)_{i = 1}^{\poseig} \).

Conversely, if \( (\dict{i}\opt)_{i = 1}^{\poseig} \) is a dictionary that satisfies \eqref{eq:rank-1-opt-dictionary-1} for some eigenbasis \( U \Let (u_i)_{i = 1}^{\poseig} \), the dictionary \( (\dictx{i}\opt)_{i = 1}^{\dsize} \Let ( U\transp \dict{i}\opt)_{i = 1}^{\dsize} \) satisfies \eqref{eq:rank-1-opt-dictionary} for \( (u_i')_{i = 1}^{\poseig} \Let (e_i)_{i= 1}^{\poseig} \). The proof is now complete.

%%%%%%%%%%%%%%%%%%%%%%%%%%%%%%%%%%
\subsection{Proof of Proposition \ref{proposition:tight-frame-and-opt-dictionaries}.}
If $V$ is uniformly distributed over the unit sphere, we have $\Sigma_V = \EE[V V\transp] = \frac{1}{n} I_n$. From Theorem \ref{theorem:DL-theorem-1} and its proof, we know that a dictionary \( (\dict{i}\opt)_{i = 1}^{\dsize} \), in this case, is \( \ell_2 \)-optimal  if and only if it satisfies:
    \[
    \sum_{i = 1}^{\dsize} \dict{i}\opt {\dict{i}\opt}\transp = \sum_{i = 1}^{\dimension} \eigval{i}\opt e_i e_i\transp,
    \]
where \( (\eigval{i}\opt)_{i = 1}^{\dimension} \) is the unique solution to the problem:
\begin{equation}
\label{eq:prop-tight-frames-1}
\begin{aligned}
		& \minimize_{ ( \eigval{i} )_i \subset \; \R }	&& \sum_{i = 1}^{\dimension} \frac{1}{\eigval{i}} \\
		& \sbjto	&& 		( \length{1}, \length{2}, \ldots, \length{\dsize} ) \majorize ( \eigval{1}, \ldots, \eigval{\dimension},0, \ldots, 0 ).
	\end{aligned}
\end{equation}
Clearly, any sequence \( (\eigval{i})_{i = 1}^{\dimension} \) that is feasible for \eqref{eq:prop-tight-frames-1} satisfies the condition \( \sum_{i = 1}^{\dimension} \eigval{i} = \sum_{i = 1}^{\dsize} \length{i} \), and thus, the optimal value of \eqref{eq:prop-tight-frames-1} is bounded below by that of:
\begin{equation}
\label{eq:prop-tight-frames-2}
\begin{aligned}
		& \minimize_{ ( \eigval{i} )_i \subset \; \R }	&& \sum_{i = 1}^{\dimension} \frac{1}{\eigval{i}} \\
		& \sbjto	&& 		\sum_{i = 1}^{\dimension} \eigval{i} = \sum_{i = 1}^{\dsize} \length{i}.
\end{aligned}
\end{equation}
If a sequence \( (\eigval{i}\opt)_{i = 1}^{\dimension} \) is defined by \( \eigval{i}\opt \Let \frac{1}{\dimension} \sum_{i = 1}^{\dsize} \length{i} \) for all \( i = 1, 2, \ldots, \dimension \), then it can be easily verified that the sequence \( (\eigval{i}\opt)_{i = 1}^{\dimension} \) is an optimal solution to \eqref{eq:prop-tight-frames-2}. Moreover, if \( \length{1} \leq \frac{1}{\dimension} \sum_{i = 1}^{\dsize} \length{i} \), it is easy to see that the sequence \( (\eigval{i}\opt)_{i = 1}^{\dimension} \) is feasible for \eqref{eq:prop-tight-frames-1}, and therefore, solves \eqref{eq:prop-tight-frames-1}.

Thus, a dictionary \( (\dict{i}\opt)_{i = 1}^{\dsize} \) is \( \ell_2 \)-optimal if and only if it satisfies
\[
 \sum_{i = 1}^{\dsize} \dict{i}\opt {\dict{i}\opt}\transp = \sum_{i = 1}^{\dimension} \eigval{i}\opt e_i e_i\transp = \left( \frac{1}{\dimension} \sum_{i = 1}^{\dsize} \length{i} \right) \sum_{i = 1}^{\dimension} e_i e_i\transp = \left( \frac{1}{\dimension} \sum_{i = 1}^{\dsize} \length{i} \right) I_n.
\]
Equivalently, \( (\dict{i}\opt)_{i = 1}^{\dsize} \) is \( \ell_2 \)-optimal if and only if it is a tight frame.

%%%%%%%%%%%%%%%%%%%%%%

\subsection{Proof of Proposition \ref{proposition:robustness-wrt-estimates}}
When \( \mu' \) and \( \sigmat{}' \) are estimated using sufficiently large number of samples, we know that
\[
X_{\randomvec} (\mu') = X_{\randomvec} (\mu) = \image (\sigmat{}\opt) = \image (\sigmat{}') = \Span (\dict{i}')_{i = 1}^{\dsize}.
\]
Let \( (\eigvalsig{i}')_{i = 1}^{\poseig} \) be the non-zero eigenvalues of \( \sigmat{}' \) arranged in non-increasing order and let \( (u_i')_{i = 1}^{\poseig} \) be the corresponding sequence of eigenvectors. From the proofs of Theorem \ref{theorem:DL-theorem-1} and \ref{theorem:DL-theorem-2} we infer that the computed dictionary satisfies 
\[
\sum_{i = 1}^{\dsize} \dict{i}' {\dict{i}'}\transp = \sum_{i = 1}^{\poseig} \eigval{i}' u'_i {u_i'}\transp,
\]
where \( (\eigval{i}')_{i = 1}^{\poseig} \) is the unique optimal solution to the problem:
\begin{equation*}
\begin{aligned}
		& \minimize_{ ( \eigval{i} )_i }	&& \sum_{i = 1}^{\poseig} \frac{\eigvalsig{i}' }{\eigval{i}} \\
		& \sbjto	&& 		\begin{cases} 
		              			 0 < \eigval{\poseig} \leq \cdots \leq \eigval{1}, \\
		              			 ( \length{1}, \length{2}, \ldots, \length{\dsize} ) \majorize ( \eigval{1}, \ldots, \eigval{\poseig},0, \ldots, 0 ). 
		                    \end{cases}
\end{aligned}
\end{equation*}
It can be easily verified that the cost of representation \( J(\mu', \sigmat{}') \) satisfies 
\[
J(\mu', \sigmat{}') = \EE_{\mathsf{P}} [\norm{\scheme'(V)}^2] = \trace \left( \sigmat{\randomvec - \mu'} \left( \sum_{i = 1}^{\poseig} \frac{1}{\eigval{i}'} u_i' {u_i'}\transp \right) \right) = \sum_{i = 1}^{\poseig} \frac{{u_i'}\transp \sigmat{\randomvec - \mu'} u_i'}{\eigval{i}'}.
\]

\begin{itemize}[leftmargin=*]
\item Letting \( (\eigvalsig{i}\opt)_{i = 1}^{\poseig} \) denote the non-zero eigenvalues of \( \sigmat{}\opt \), and \( (\eigval{i}\opt)_{i = 1}^{\poseig} \) denote the unique optimal solution to:
\begin{equation*}
\begin{aligned}
		& \minimize_{ ( \eigval{i} )_i }	&& \sum_{i = 1}^{\poseig} \frac{\eigvalsig{i}\opt }{\eigval{i}} \\
		& \sbjto	&& 		\begin{cases} 
		              			 0 < \eigval{\poseig} \leq \cdots \leq \eigval{1}, \\
		              			 ( \length{1}, \length{2}, \ldots, \length{\dsize} ) \majorize ( \eigval{1}, \ldots, \eigval{\poseig},0, \ldots, 0 ),
		                    \end{cases}
\end{aligned}
\end{equation*}
we observe that \( \eigval{i}' \xrightarrow[(\eigvalsig{i}')_{i = 1}^{\poseig} \rightarrow (\eigvalsig{i}\opt)_{i = 1}^{\poseig} ]{} \eigval{i}\opt \) for all \( i = \indexes{\poseig} \).

\item We also readily observe that 
\begin{align*}
{u_i'}\transp \big( \sigmat{\randomvec - \mu'} \big) u_i' & = {u_i'}\transp \big( \sigmat{\randomvec - \mu'} + \sigmat{}' - \sigmat{}' + \sigmat{}\opt - \sigmat{}\opt \big) u_i' \\
& = {u_i'}\transp \big( \sigmat{}' \big) u_i' + {u_i'}\transp \big( \sigmat{\randomvec - \mu'} - \sigmat{}\opt \big) u_i' + {u_i'}\transp \big( \sigmat{}\opt - \sigmat{}' \big) u_i' \\
& = \eigvalsig{i}' + {u_i'}\transp \big( \sigmat{\randomvec - \mu'} - \sigmat{}\opt \big) u_i' + {u_i'}\transp \big( \sigmat{}\opt - \sigmat{}' \big) u_i'.
\end{align*}
Since \( (\mu', \sigmat{}') \longrightarrow (\mu, \sigmat{}\opt) \), we see that \( \eigvalsig{i}' \longrightarrow \eigvalsig{i}\opt \), \( \big( \sigmat{\randomvec - \mu'} - \sigmat{}\opt \big) \longrightarrow 0 \) and \( \big( \sigmat{}\opt - \sigmat{}' \big) \longrightarrow 0 \).
\end{itemize}

Therefore, we see that 
\[
J(\mu;, \sigmat{}') = \sum_{i = 1}^{\poseig} \frac{{u_i'}\transp \sigmat{\randomvec - \mu'} u_i'}{\eigval{i}'} \xrightarrow[(\mu', \sigmat{}') \longrightarrow (\mu', \sigmat{}\opt) ]{} \sum_{i = 1}^{\poseig} \frac{\eigvalsig{i}\opt}{\eigval{i}\opt} = J(\mu, \sigmat{}\opt),
\]
and the assertion follows.

%%%%%%%%%%%%%%%%%%%%%%%%%%%%%%%%%%%%%%%%%%%%%%%%%%%%%%%%%%%%%%%%%%%%%%%%%%%%%%%%%%%%%%%%%%%%%%%%%%
\section{Proofs of auxiliary results}
\label{section:proof-of-auxiliary-results}
\subsection{Proof of Lemma \ref{lemma: orthonormal basis}}
It is clear that an orthonormal basis \( (\orthobasis{i})_{i = 1}^{\dsize} \) of \( \R^{\dsize} \) with the property
\[
\inprod{\orthobasis{i}}{\linearmap \orthobasis{i}} = \length{i} \quad \text{for all \( i = \indexes{\dsize} \),}
\]
exists if and only if 
\begin{equation}
\label{eq: orthonormal basis property}
\inprod{\orthobasis{i}}{\frac{1}{2} ( \linearmap + \linearmap^{\transp} ) \orthobasis{i}} = \length{i} \quad \text{for all \( i = \indexes{\dsize} \).}
\end{equation}
Let \( \ortho{\dsize} \ni X \lmt S(X) \in \symmetric{\dsize} \) be defined by
\begin{equation}
\label{eq: S definition}
S(X) \Let X^{\transp} \bigg( \frac{1}{2} ( \linearmap + \linearmap^{\transp} ) \bigg) X.
\end{equation}
Since for every \( X \in \ortho{\dsize} \) the matrix \( S(X) \) is a conjugation of \( \frac{1}{2} ( \linearmap + \linearmap^{\transp} ) \), we conclude that the eigenvalues of \( S(X) \) are \( (\eigval{i})_{i = 1}^{\dsize} \). Conversely, for every \( S' \in \symmetric{\dsize} \) having eigenvalues \( (\eigval{i})_{i = 1}^{\dsize} \), there exists an \( X' \in \ortho{\dsize}\) such that \( S' = S(X') \).

From the \emph{Schur-Horn Theorem} \ref{theorem: schur horn} we know that there exists a symmetric matrix \( S \in \symmetric{\dsize} \) with eigenvalues \( (\eigval{i})_{i = 1}^{\dsize} \) and diagonal entries \( (\length{i})_{i = 1}^{\dsize} \) if and only if \( (\length{i})_{i = 1}^{\dsize} \majorize (\eigval{i})_{i = 1}^{\dsize} \). 
In other words, there exists an orthogonal matrix \( X \in \ortho{\dsize} \) such that \( S(X) \) has diagonal entries equal to \( (\length{i})_{i = 1}^{\dsize} \) if and only if \( (\length{i})_{i = 1}^{\dsize} \majorize (\eigval{i})_{i = 1}^{\dsize} \). Denoting the \( i \)th column of \( X \) by \( \orthobasis{i} \), we see that \( (\orthobasis{i})_{i = 1}^{\dsize} \) is an orthonormal basis of \( \R^{\dsize} \) and 
\[ 
\inprod{\orthobasis{i}}{\frac{1}{2} ( \linearmap + \linearmap^{\transp} ) \orthobasis{i}} = \big( S(X) \big)_{ii} \quad \text{is the \( i^{\text{th}} \) diagonal entry.}
\]
Therefore, we conclude that there exists an orthonormal basis \( (\orthobasis{i})_{i = 1}^{\dsize} \) of \( \R^{\dsize} \) that satisfies \eqref{eq: orthonormal basis property} if and only if \( (\length{i})_{i = 1}^{\dsize} \majorize (\eigval{i})_{i = 1}^{\dsize} \).

We now prove that the procedure described in Algorithm \ref{algo: orthonormal basis} indeed computes the required orthonormal basis. Let \( k \leq \dsize \) be a positive integer. A sequence of vectors \( (\uvec{i})_{i = 1}^k \) in \( \R^{\dsize} \) and a non-increasing sequence of real numbers \( (\aval{i})_{i = 1}^k \) are jointly said to be \emph{valid} if they satisfy the following properties:
\begin{enumerate}
\item The sequence of vectors \( (\uvec{i})_{i = 1}^k \) is an ordered collection of orthonormal vectors in \( \R^{\dsize} \) such that the sequence of real numbers \( (\inprod{\uvec{i}}{\linearmap \uvec{i}})_{i = 1}^k \) is in non-increasing order.

\item Whenever \( k \geq 2 \), \( p,q \in \{ \indexes{k} \} \) and \( p \neq q \), we have 
\[
\inprod{\uvec{p}}{(\linearmap + \linearmap^{\transp}) \uvec{q}} = 0.
\]

\item \( (\aval{i})_{i = 1}^k \majorize (\inprod{\uvec{i}}{\linearmap \uvec{i}})_{i = 1}^k \).
\end{enumerate}
Observe that such a collection always exists, for instance, any sub collection of eigenvectors of cardinality \( k \) and sequence of real numbers majorized by the sequence of corresponding eigenvalues would satisfy all the properties mentioned here.

Let us compute a unit vector \( x \in \R^{\dsize} \)  that is in the linear span of the vectors \( (\uvec{i})_{i = 1}^k \) and satisfies 
\[
\inprod{x}{\linearmap x} = \aval{1},
\]
via the following steps :

Let \( (\bval{i})_{i = 1}^k \) be a non-increasing sequence of real numbers defined by \( \bval{i} \Let \inprod{\uvec{i}}{\linearmap \uvec{i}} \) for all \( i = \indexes{k} \). Since \( (\aval{i})_{i = 1}^k \majorize (\bval{i})_{i = 1}^k \) by property (3), we have \( \bval{1} \geq \aval{1} \).
\begin{itemize}[leftmargin=*]
\item \textbf{case 1:} If \( \bval{1} = \aval{1} \), let us define
\begin{align*}
\begin{cases}
x & \Let \uvec{1}, \\
(\uvec{i}')_{i = 1}^{k - 1} & \Let (\uvec{(i + 1)})_{i = 1}^{k - 1}, \\
(\aval{i}')_{i = 1}^{k - 1} & \Let (\aval{(i + 1)})_{i = 1}^{k - 1}.
\end{cases}
\end{align*}
We see that the vector \( x \) satisfies
\[
\inprod{x}{\linearmap x} = \inprod{\uvec{1}}{\linearmap \uvec{1}} = \bval{1} = \aval{1},
\]
and \( \inprod{\uvec{i}'}{x} = 0 \) for all \( i = \indexes{k - 1} \). Also one can easily verify that \( (\uvec{i}')_{i = 1}^{k - 1} \) is an orthonormal sequence of vectors in \( \R^{\dsize} \) and \( \inprod{\uvec{p}'}{(\linearmap + \linearmap^{\transp}) \uvec{q}'} = 0 \) for all \( p,q \in \{ \indexes{k - 1} \}\) with \( p \neq q \). The condition that \( (\aval{i})_{i = 1}^k \majorize (\bval{i})_{i = 1}^k \) along with the fact that \( \aval{1} = \bval{1} \), implies at once that \( (\aval{i})_{i = 2}^{k} \majorize (\bval{i})_{i = 2}^{k} \). In other words, we have \( (\aval{i}')_{i = 1}^{k - 1} \majorize ( \inprod{\uvec{i}'}{\linearmap \uvec{i}'})_{i = 1}^{k - 1} \). This implies that the sequences \( (\uvec{i}')_{i = 1}^{k - 1} \), \( (\aval{i}')_{i = 1}^{k - 1} \) jointly satisfy all the properties (1),(2) and (3) described earlier and therefore are valid.

\item \textbf{case 2:} If \( \bval{1} > \aval{1} \), we know that \( (\aval{i})_{i = 1}^k \majorize (\bval{i})_{i = 1}^k \) and in particular we know that \( \sum_{i = 1}^k \aval{i} = \sum_{i = 1}^k \bval{i} \). Therefore, there exists \( i \) such that \( 1 < i \leq k \) and 
\[
 \inprod{\uvec{i}}{\linearmap \uvec{i}} = \bval{i} \quad \leq \; \aval{1} \; < \quad \bval{(i - 1)} = \inprod{\uvec{(i- 1)}}{\linearmap \uvec{(i - 1)}}.
\] 
Let us define a map \( \R \ni \theta \lmt \gmap{\theta} \in \R \) as,
\begin{equation}
\label{eq: definition g function}
\gmap{\theta} \Let \inprod{ \frac{\theta \uvec{1} + (1 - \theta) \uvec{i}}{\sqrt{\theta^2 + (1 - \theta)^2}} }{\; \linearmap \;\; \frac{\theta \uvec{1} + (1 - \theta) \uvec{i}}{\sqrt{\theta^2 + (1 - \theta)^2}}}.
\end{equation}
It should be observed that for \( \theta = 0 \) we have
\[
\gmap{0} = \inprod{ \frac{\theta \uvec{1} + (1 - \theta) \uvec{i}}{\sqrt{\theta^2 + (1 - \theta)^2}} }{\; \linearmap \;\; \frac{\theta \uvec{1} + (1 - \theta) \uvec{i}}{\sqrt{\theta^2 + (1 - \theta)^2}}} \Bigg \rvert_{\theta = 0} = \inprod{\uvec{i}}{\linearmap \uvec{i}} \leq \aval{1},
\] 
and for \( \theta = 1 \) we have
\[
\gmap{1} = \inprod{ \frac{\theta \uvec{1} + (1 - \theta) \uvec{i}}{\sqrt{\theta^2 + (1 - \theta)^2}} }{\; \linearmap \;\; \frac{\theta \uvec{1} + (1 - \theta) \uvec{i}}{\sqrt{\theta^2 + (1 - \theta)^2}}} \Bigg \rvert_{\theta = 1} = \inprod{\uvec{1}}{\linearmap \uvec{1}} > \aval{1}.
\] 
Since \( \gmap{\theta} \) is a continuous function of \( \theta \), there exists a solution \( \Theta \in  [0,1] \) for the following quadratic equation
\begin{equation}
\label{eq: g = 0}
\gmap{\theta} = \inprod{ \frac{\theta \uvec{1} + (1 - \theta) \uvec{i}}{\sqrt{\theta^2 + (1 - \theta)^2}} }{\; \linearmap \;\; \frac{\theta \uvec{1} + (1 - \theta) \uvec{i}}{\sqrt{\theta^2 + (1 - \theta)^2}}} = \aval{1}.
\end{equation}
Using the fact that \( \inprod{\uvec{1}}{\linearmap \uvec{1}} = \bval{1} \), \( \inprod{\uvec{i}}{\linearmap \uvec{i}} = \bval{i} \), and \( \inprod{\uvec{1}}{(\linearmap + \linearmap^{\transp}) \uvec{i}} = 0 \), one can simplify \eqref{eq: g = 0} to get the following quadratic equation
\[
(\bval{1} + \bval{i} - 2 \aval{1}) \theta^2 \; + \; 2 (\aval{1} - \bval{i}) \theta \; + \; (\bval{i} - \aval{i}) = 0.
\]
Simple calculations lead to a solution
\begin{equation}
\label{eq: theta solution}
\Theta \Let \frac{\sqrt{\aval{1} - \bval{i}}}{\sqrt{\aval{1} - \bval{i}} + \sqrt{\bval{1} - \aval{1}} } \; \in \; [0,1],
\end{equation}
which is the required root of \eqref{eq: g = 0}.

We define \( x \), \( v \in \R^{\dsize}\), \( (\uvec{i}')_{i = 1}^{k - 1}  \subset \R^{\dsize} \) and \( (\aval{i}')_{i = 1}^{k - 1} \subset \R \) as: 
\begin{equation}
\label{eq: update in proof}
\begin{aligned}
\begin{cases}
x & \Let \frac{\Theta \uvec{1} + (1 - \Theta) \uvec{i}}{\sqrt{\Theta^2 + (1 - \Theta)^2}}, \\
v & \Let \frac{(1 - \Theta) \uvec{1} - \Theta \uvec{i}}{\sqrt{\Theta^2 + (1 - \Theta)^2}}, \\
(\uvec{i}')_{i = 1}^{k - 1} & \Let \sort \big\{ (\uvec{i})_{i = 1}^k \setminus \{\uvec{1}, \uvec{i} \} \cup \{ v \}  \big\}, \\
(\aval{i}')_{i = 1}^{k - 1} & \Let (\aval{i + 1})_{i = 1}^{k - 1}.
\end{cases}
\end{aligned}
\end{equation}

It is immediate that \( \inprod{x}{\linearmap x} = \gmap{\Theta} = \aval{1} \) and \( \inprod{v}{x} = 0 \). Since \( x \) is a linear combination of \( \uvec{1} \) and \( \uvec{i} \), it also follows that for all \( p = 2 ,\ldots, (i - 1), (i + 1), \ldots, k \), we have \( \inprod{\uvec{p}}{x} = 0 \). We conclude that \( \inprod{\uvec{i}'}{x} = 0 \) for all \( i = \indexes{k - 1} \).

Since \( v \) is a linear combination of the vectors \( \{ \uvec{1}, \uvec{i} \} \), we have \( \inprod{\uvec{p}}{v} = 0 \) and \( \inprod{\uvec{p}}{(\linearmap + \linearmap^{\transp}) v} = 0 \) for all \( p = 2 ,\ldots, (i - 1), (i + 1), \ldots, k \). It also follows immediately that \( \inprod{\uvec{p}}{\uvec{q}} = 0 \) and \( \inprod{\uvec{p}}{(\linearmap + \linearmap^{\transp}) \uvec{q}} = 0 \) for all \( p,q \in \{ 2,\ldots,(i - 1),(i + 1),\ldots,k \} \) with \( p \neq q \). This implies that \( (\uvec{i}')_{i = 1}^{k - 1} \) is an orthonormal sequence of vectors in \( \R^{\dsize} \) that also satisfies property (2).

Now we shall prove that \( (\aval{i}')_{i = 1}^{k - 1} \majorize ( \inprod{\uvec{i}'}{\linearmap \uvec{i}'} )_{i = 1}^{k - 1} \). Since \( \Span \{ ( \uvec{i}' )_{i = 1}^{k - 1} \cup ( x ) \} = \Span \{ ( \uvec{i} )_{i = 1}^k \} \), it is easily verified that \( \inprod{v}{\linearmap v} = (\bval{1} + \bval{i} - \aval{1}) \). Defining \( \bval{i}' \Let \inprod{\uvec{i}'}{\linearmap \uvec{i}'} \) for all \( i = \indexes{k - 1} \), we see at once that 
\[
(\bval{i}')_{i = 1}^{k - 1} = \sort \big\{ \{ \bval{i} \}_{i = 1}^k \setminus \{ \bval{1},\bval{i} \} \cup \{ \bval{1} + \bval{i} - \aval{1} \}  \big\}.
\] 
For \( \bval{1} + \bval{i} - \aval{1} \), exactly one of the following two cases arise:
\begin{itemize}[label=\(\circ\), leftmargin=*]
\item \( \bval{1} + \bval{i} - \aval{1} \leq \bval{i - 1 }  \). 

In this case we have 
\begin{align*}
\bval{j}' & = \bval{i + 1} \quad \text{for all \( j = \indexes{i - 2} \),} \\
\bval{i - 1}' & = (\bval{1} + \bval{i} - \aval{1}), \\
\bval{j}' & = \bval{j + 1} \quad \text{for all \( j = i, \ldots, k - 1 \)}.
\end{align*}
For \( j,l \) such that \( 1 \leq l \leq i - 2 \) and \( 1 \leq j \leq l \), we see that \( \bval{j}' \geq \bval{1} + \bval{i} - \aval{1} > \aval{1} \). By definition of \( \bval{i} \), it follows that \( \bval{j}' = \bval{j + 1} > \aval{1} \). Therefore 
\[
\sum_{j = 1}^l \bval{j}' > \sum_{j = 1}^l \aval{1} \geq \sum_{j = 2}^{l + 1} \aval{j} = \sum_{j = 1}^l \aval{j}'. 
\]
For \( l = i - 1 \), we have
\begin{align*}
\sum_{j = 1}^{i - 1} \bval{j}' & = \sum_{j = 2}^{i - 1} \bval{j} + \bval{1} + \bval{i} - \aval{1} \\
& = \sum_{j = 1}^{i} \bval{j} - \aval{1} \geq \sum_{j = 1}^{i} \aval{j} - \aval{1} \\
& \geq \sum_{j = 1}^{i - 1} \aval{j}'.
\end{align*}

\item \( \bval{i - 1} < (\bval{1} + \bval{i} - \aval{1}) \).

In this case we see that \( (\bval{j}')_{j = 1}^{i - 2} = \sort \{ (\bval{j})_{j = 2}^{i - 2} \cup ((\bval{1} + \bval{i} - \aval{1})) \} \), \( \bval{(i - 1)}' = \bval{i - 1} \) and \( \bval{j}' = \bval{j + 1} \) for all \( j = i, \ldots, k - 1 \). For \( l \) such that \( 1 \leq l \leq (i - 1) \), 
\[
\sum_{j = 1}^l \bval{j}' > \sum_{j = 1}^l \bval{i - 1}' > \sum_{j = 1}^l \aval{1} \geq \sum_{j = 2}^{l + 1} \aval{j} = \sum_{j = 1}^l \aval{j}'.
\]
\end{itemize}
Finally, for \( i \leq l \leq k - 1 \), we have 
\begin{align*}
\sum_{j = 1}^l \bval{j}' & = \sum_{j = 2}^{i - 1} \bval{j} + \bval{1} + \bval{i} - \aval{1} + \sum_{j = i + 1}^{l + 1} \bval{j} \\
& = \sum_{j = 1}^{l + 1} \bval{j} - \aval{1} \geq \sum_{j = 2}^{l + 1} \aval{j} \\
& \geq \sum_{j = 1}^l \aval{j}'.
\end{align*}
Therefore, we have proved that \( (\aval{i}')_{i = 1}^{k - 1} \majorize (\bval{i}')_{i = 1}^{k - 1} = (\inprod{\uvec{i}'}{\linearmap \uvec{i}'})_{i = 1}^{k - 1} \).
\end{itemize}

To summarize, we started with \emph{valid} sequences \( (\uvec{i})_{i = 1}^k \) and \( (\aval{i})_{i = 1}^k \), from these, two new similarly valid sequences \( (\uvec{i}')_{i = 1}^{k - 1} \), \( (\aval{i}')_{i = 1}^{ k- 1} \) were computed along with a vector \( x \) that satisfies \( \inprod{x}{\linearmap x} = \aval{1} \) and \( \inprod{\uvec{i}'}{x} = 0 \) for all \( i = \indexes{k - 1} \). 

It has to be noted that the sequences \( (\uvec{i}(1))_{i = 1}^{\dsize} \) and \( (\aval{i}(1))_{i = 1}^{\dsize} \) are valid, i.e., they jointly satisfy the properties (1),(2) and (3). Let us assume that for some \( t \in \{ \indexes{\dsize - 1} \} \), the sequences \( (\uvec{i}(t))_{i = 1}^{\dsize - t + 1} \) and \( (\aval{i}(t))_{i = 1}^{\dsize - t + 1} \) are valid. From the analysis done above, we conclude that by following the procedure given in an iteration of the loop (step) of Algorithm \ref{algo: orthonormal basis}, two new valid sequences \( (\uvec{i}(t + 1))_{i = 1}^{\dsize - t} \), \( (\aval{i}(t + 1))_{i = 1}^{\dsize - t} \) are obtained and in addition, we get a vector \( \orthobasis{t} \in \Span (\uvec{i}(t))_{i = 1}^{\dsize - t + 1} \) that satisfies 
\begin{align*}
\inprod{\orthobasis{t}}{\linearmap \orthobasis{t}} & = \aval{1}(t) = \aval{t}(1) = \length{t}, \\
\inprod{\uvec{i}(t + 1)}{\orthobasis{t}} & = 0 \quad \text{for all \( i = \indexes{\dsize - t} \).}
\end{align*}
Using induction on \( t \) and the fact that \( \orthobasis{t} \in \Span (\uvec{i}(t))_{i = 1}^{\dsize - t + 1} \) and \( \inprod{\uvec{i}(t + 1)}{\orthobasis{t}} = 0 \) for all \( i = \indexes{\dsize - t} \), we get that \( (\orthobasis{t})_{t = 1}^{\dsize} \) is a sequence of orthonormal vectors in \( \R^{\dsize} \). We conclude that the Algorithm \ref{algo: orthonormal basis} indeed computes an orthonormal basis \( (\orthobasis{t})_{t = 1}^{\dsize} \) with properties described in Lemma \ref{lemma: orthonormal basis}. The proof is now complete

\subsection{Proof of Lemma \ref{lemma: primal dual}}
We shall begin with calculating the Lagrange dual of \eqref{eq: primalgeneral}. Let \( \eqmul \) be the KKT multiplier for the equality constraint and \( (\ineqmul{j})_{j = 1}^{\poseig - 1} \) be the KKT multipliers for the inequality constraints in \eqref{eq: primalgeneral}. The Lagrangian \( \lagrangian \left( (\eigval{i})_{i = 1}^{\poseig}, \eqmul, (\ineqmul{j})_{j = 1}^{\poseig - 1} \right) \) is defined by 
\begin{equation}
\label{eq: lagrangian}
\begin{aligned}
\lagrangian \left( (\eigval{i})_{i = 1}^{\poseig}, \eqmul, (\ineqmul{j})_{j = 1}^{\poseig - 1} \right) & \Let \sum_{i = 1}^{\poseig} s_i \Big( \frac{1}{\eigval{i}} \Big) + \eqmul \Big( \sum_{i = 1}^{\poseig} \eigval{i} - a_i \Big) + \sum_{j = 1}^{\poseig - 1} \ineqmul{j} \Big( \sum_{i = 1}^j a_i - \eigval{i} \Big), \\
& = \sum_{i = 1}^{\poseig} s_i \Big( \frac{1}{\eigval{i}} \Big) + \eigval{\poseig} \eqmul + \sum_{i = 1}^{\poseig - 1} \eigval{i} \Big( \eqmul - \sum_{j = i}^{\poseig - 1} \ineqmul{j} \Big) \\
& - \Big( a_{\poseig} \eqmul + \sum_{i = 1}^{\poseig - 1} a_i \big( \eqmul - \sum_{j = i}^{\poseig - 1} \ineqmul{j} \big) \Big).
\end{aligned}
\end{equation}
Let \( g(\eqmul, (\ineqmul{j})_{j = 1}^{\poseig - 1}) \) be the Lagrange dual function given by 
\begin{equation}
\label{eq: lagrange dual function}
g(\eqmul, (\ineqmul{j})_{j = 1}^{\poseig - 1}) \; \Let \; \min_{\eigval{i} > 0, \; i = 1, \ldots, \poseig} \quad \lagrangian  \left( (\eigval{i})_{i = 1}^{\poseig}, \eqmul, (\ineqmul{j})_{j = 1}^{\poseig - 1} \right).
\end{equation}
It can be seen that for \( \eqmul < 0 \) one can select\( \eigval{i} = 1 \) for all \( i = \indexes{\poseig - 1} \), and  \( \eigval{\poseig} \) arbitrarily large, whereby the minimum value achieved in \eqref{eq: lagrange dual function} is negative infinity. Similarly for any \( i \in \{ \indexes{\poseig - 1} \}\), if \( \eqmul - \sum_{j = i}^{\poseig - 1} \ineqmul{j} < 0 \), one can select \( \eigval{l} = 1 \) for all \( l \neq i \) and \( \eigval{i} \) arbitrarily large, leading to the minimum value being negative infinity again. When \( \eqmul \geq 0 \) and \( \eqmul - \sum_{j = i}^{\poseig - 1} \ineqmul{j} \geq 0 \) for all \( i = \indexes{\poseig - 1} \), due to convexity of \( \lagrangian \) the minimizer in \eqref{eq: lagrange dual function} can be calculated by equating the gradient of \( \lagrangian \) to zero. This minimizer is given by
\begin{align*}
\eigval{i} &= \sqrt{\frac{s_i}{\eqmul - \sum_{j = 1}^{\poseig - 1} \ineqmul{j}}} \quad \text{for all \( i = \indexes{\poseig - 1}, \)} \\
\eigval{\poseig} &= \sqrt{\frac{s_{\poseig}}{\eqmul}},
\end{align*}
with the optimal value given by 
\begin{equation}
\begin{aligned}
g'(\eqmul, (\ineqmul{j})_{j = 1}^{\poseig - 1}) \Let & \sum_{i = 1}^{\poseig - 1} \big( \eqmul - \sum_{j = i}^{\poseig - 1} \ineqmul{j} \big)^{1/2} \Big( 2 \sqrt{s_i} - a_i \big( \eqmul - \sum_{j = i}^{\poseig - 1} \ineqmul{j} \big)^{1/2} \Big) \\
& + (\eqmul)^{1/2} \big( 2 \sqrt{s_{\poseig}} \; - \; a_{\poseig} (\eqmul)^{1/2} \big).
\end{aligned}
\end{equation}
Therefore, the Lagrange dual function is
\[
g(\eqmul, (\ineqmul{j})_{j = 1}^{\poseig - 1}) = 
\begin{cases}
g'(\eqmul, (\ineqmul{j})_{j = 1}^{\poseig - 1}) & \text{\( \eqmul \geq 0, \eqmul - \sum_{j = i}^{\poseig - 1} \ineqmul{j} \geq 0 \)}, \\
-\infty & \text{otherwise}.
\end{cases}
\]
The Lagrange dual problem of \eqref{eq: primalgeneral} is 
\begin{equation}
\label{eq:lagrange dual problem}
\begin{aligned}
& \maximize_{\eqmul, (\ineqmul{j})_{j = 1}^{\poseig - 1}}	&&  g'(\eqmul, (\ineqmul{j})_{j = 1}^{\poseig - 1}) \\
		& \sbjto							&&
		\begin{cases}
			\eqmul \geq 0, \\
			\eqmul - \sum_{j = i}^{\poseig - 1} \ineqmul{j} \geq 0 \quad \text{for all \(i = \indexes{\poseig - 1}\)}, \\
			\ineqmul{j} \geq 0 \quad \text{for all \(j = \indexes{\poseig - 1}\)}.
		\end{cases}
\end{aligned}
\end{equation}
We define a new set of variables \( (\dualvar{t})_{t = 1}^{\poseig} \) by
\begin{equation}
\label{eq:variable transformation in primal dual}
\begin{aligned}
\dualvar{t} & \Let \Big( \eqmul - \sum_{j = t}^{\poseig - 1} \ineqmul{j} \Big)^{1/2} \quad \text{for all \( t = \indexes{\poseig - 1} \)}, \\
\dualvar{\poseig} & \Let (\eqmul)^{1/2}.
\end{aligned}
\end{equation}
It is easy to see that the mapping \( \big( \eqmul, (\ineqmul{j})_{j = 1}^{\poseig - 1} \big) \lmt (\dualvar{t})_{t = 1}^{\poseig} \) given by \eqref{eq:variable transformation in primal dual} is injective whenever it is well defined. The first two constraints in \eqref{eq:lagrange dual problem} imply that \( \dualvar{t} \geq 0 \) for all \( t = \indexes{\poseig} \), and the third constraint implies that \( \dualvar{1} \leq \dualvar{2} \leq \cdots \leq \dualvar{\poseig} \). Conversely for a sequence  of numbers \( (\dualvar{t})_{t = 1}^{\poseig} \) such that \( 0 \leq \dualvar{1} \leq \cdots \leq \dualvar{\poseig} \), it is clear that the preimage \( \big( \eqmul, (\ineqmul{j})_{j = 1}^{\poseig - 1} \big) \) also satisfies the constraints of \eqref{eq:lagrange dual problem}. 
The objective function in \eqref{eq:lagrange dual problem} in terms of the new variables \( (\dualvar{t})_{t = 1}^{\poseig} \) is given by
\[
\sum_{t = 1}^{\poseig} 2 \sqrt{s_t} \dualvar{t} - a_t \dualvar{t}^2. 
\]
Therefore, the problem \eqref{eq:lagrange dual problem} can be recast in terms of variables \( (\dualvar{t})_{t = 1}^{\poseig} \) as a minimization problem in the following way:
\begin{equation}
\label{eq:modified lagrange dual}
\begin{aligned}
& \minimize_{( \dualvar{t} )_{t = 1}^{ \poseig }}	&&  \sum_{t = 1}^{\poseig} a_t \dualvar{t}^2 - 2 \sqrt{s_t} \dualvar{t} \\
		& \sbjto							&&
			0 \leq \dualvar{1} \leq \cdots \leq \dualvar{\poseig}.
\end{aligned}
\end{equation} 
Since \( a_t > 0 \) for all \( t = \indexes{\poseig} \), the optimization problem \eqref{eq:modified lagrange dual} is a convex quadratic program and admits an optimal solution. Let \( (\dualvar{t} \opt)_{t = 1}^{\poseig} \) be an optimal solution to \eqref{eq:modified lagrange dual} and \( q \opt \) be the optimal value. Let \( \big( \eqmul \opt, (\ineqmul{j} \opt)_{j = 1}^{\poseig - 1} \big) \) be the unique preimage of \( (\dualvar{t} \opt)_{t = 1}^{\poseig} \) obtained via \eqref{eq:variable transformation in primal dual}. Then \( \big( \eqmul \opt, (\ineqmul{j} \opt)_{j = 1}^{\poseig - 1} \big) \) is an optimal solution to \eqref{eq:lagrange dual problem} and the optimal value is \( - q \opt \). Due to strong duality between \eqref{eq: primalgeneral} and \eqref{eq:lagrange dual problem}, we conclude that the optimization problem \eqref{eq: primalgeneral} admits an optimal solution \( (\eigval{i} \opt)_{i = 1}^{\poseig} \). If we denote \( p \opt \) to be the optimal value in \eqref{eq: primalgeneral} then we have \( p \opt = - q \opt \). Because of strong duality, we know that a primal optimal solution \( (\eigval{i} \opt)_{i = 1}^{\poseig} \) and a dual solution \( \big( \eqmul \opt, (\ineqmul{j} \opt)_{j = 1}^{\poseig - 1} \big) \) together will satisfy the KKT conditions, which are:
\begin{equation}
\label{eq:kkt conditions}
\begin{aligned}
\nabla_{\eigval{i}} \lagrangian \left( (\eigval{i} \opt)_{i = 1}^{\poseig}, \eqmul \opt, (\ineqmul{j} \opt)_{j = 1}^{\poseig - 1} \right) & = 0 \quad \text{for all \( i = \indexes{\poseig} \)}, \\
\ineqmul{j} \opt \big( \sum_{i = 1}^j a_i - \eigval{i} \opt \big) & = 0 \quad \text{for all \( j = \indexes{\poseig - 1} \)}, \\
\ineqmul{j} \opt & \geq 0 \quad \text{for all \( j = \indexes{\poseig - 1} \)}.
\end{aligned}
\end{equation}
The first KKT condition can be written as 
\begin{equation}
\label{eq:dummy kkt equations}
\begin{aligned}
0 & =  \frac{- s_{\poseig}}{ ( \eigval{\poseig} \opt )^2} + \eqmul \opt , \\
0 & =  \frac{- s_i}{ ( \eigval{i} \opt )^2} + \Big( \eqmul \opt - \sum_{j = i}^{\poseig - 1} \ineqmul{j} \opt \Big) \quad \text{for all \( i = \indexes{\poseig - 1}. \)}
\end{aligned}
\end{equation}
In terms of \( (\dualvar{t} \opt)_{t = 1}^{\poseig} \) the conditions \eqref{eq:dummy kkt equations} are simplified to
\begin{equation}
\label{eq:kkt gradient conditions}
\begin{aligned}
\eigval{i} \opt & = \frac{\sqrt{s_i}}{\dualvar{i} \opt} \quad \text{for all \( i = \indexes{\poseig}. \)}
\end{aligned}
\end{equation}
Since \( a_t, s_t > 0 \) for all \( t = \indexes{\poseig} \), we conclude that the optimal solution \( (\dualvar{t} \opt)_{t = 1}^{\poseig} \) to \eqref{eq:modified lagrange dual} satisfies, \( \dualvar{t} \opt > 0 \) for all \( t = \indexes{\poseig} \), which makes the condition \eqref{eq:kkt gradient conditions} valid. It is easy to see that if \( (\dualvar{t} \opt)_{t = 1}^{\poseig} \) is an optimal solution to \eqref{eq:modified lagrange dual}, then there exists a unique solution \( (\eigval{i} \opt)_{i = 1}^{\poseig} \) that satisfies \eqref{eq:kkt gradient conditions} and vice versa. Therefore, we conclude that both the problems \eqref{eq: primalgeneral} and \eqref{eq: dualgeneral} (which is same as \eqref{eq:modified lagrange dual}), admit unique optimal solutions.

Let us define a partition \( (n_1, n_2, \ldots, n_T) \) of the set \( (\indexes{\poseig}) \) by
\begin{equation}
\label{eq:partition definition}
\begin{aligned}
\partition{1} & \Let 1, \\
        \partition{l} & \Let \min \{ t \; \vert \; \partition{(l - 1)} < t \leq \poseig, \; \dualvar{(t - 1)} \opt < \dualvar{t} \opt \} \quad  \text{for all \( l = 2,\ldots, T \).} 
\end{aligned}
\end{equation}
In addition, if we let \( n_{T + 1} \Let \poseig + 1 \), it is obvious that 
\begin{equation}
\label{eq:xt equal to constant}
\dualvar{t} \opt = c_l \quad \text{for all \( n_l \leq t < n_{l + 1} \) and \( l = \indexes{T} \), }
\end{equation}
for some constant \( c_l \). It is also obvious to see from \eqref{eq:variable transformation in primal dual} and \eqref{eq:partition definition} that among \( (\ineqmul{j} \opt)_{j = 1}^{\poseig - 1} \) we have \( \ineqmul{n_l - 1} \opt > 0 \) for \( l = 2, \ldots, T \) and the rest are equal to zero. Then the complimentary slackness condition i.e., the second KKT condition in \eqref{eq:kkt conditions} implies that 
\[
\sum_{i = 1}^{n_l - 1} a_i - \eigval{i} \opt = 0 \quad \text{for all \( l = 2, \ldots, T \),}
\]
and along with the fact that \( \sum_{i = 1}^{\poseig} a_i - \eigval{i} \opt = 0 \), we get
\begin{equation}
\label{eq:sum of partitions}
\sum_{i = n_l}^{n_{l + 1} - 1} a_i = \sum_{i = n_l}^{n_{l + 1} - 1} \eigval{i} \opt \quad \text{for all \( l = \indexes{T}. \)}
\end{equation}

Therefore, from \eqref{eq:sum of partitions}, \eqref{eq:kkt gradient conditions}, and \eqref{eq:xt equal to constant} we conclude that for \( l = \indexes{T} \) and \( n_l \leq i < n_{l + 1} \) we have 
\begin{equation}
\label{eq:primal-solution}
\eigval{i} \opt = \sqrt{s_i} \; \left( \frac{\sum\limits_{j = n_l}^{n_{l + 1} - 1} a_j}{\sum\limits_{j = n_l}^{n_{l + 1} - 1} \sqrt{s_j}} \right).
\end{equation}
The solution given in \eqref{eq:primal-solution} can be compactly written using the mapping defined in \eqref{eq:definition-of-lambda} as 
\[
(\eigval{i} \opt)_{i = 1}^{\poseig} = \eigvalmap{(\sqrt{s_i})_{i = 1}^{\poseig}}{(a_i)_{i = 1}^{\poseig}}{(\partition{l})_{l = 1}^{T}}.
\]
Using the solution \eqref{eq:primal-solution} the optimal value in the optimization problem \eqref{eq: primalgeneral} can be readily calculated, and it is given by
\begin{equation}
\label{eq:primal-optimal-value}
-q \opt = p \opt = \sum\limits_{l = 1}^T \left( \frac{ \left( \sum\limits_{j = n_l}^{n_{l + 1} - 1} \sqrt{s_j} \right)^2 }{\sum\limits_{j = n_l}^{n_{l + 1} - 1} a_j} \right) = \optval{(\sqrt{s_i})_{i = 1}^{\poseig}}{(a_i)_{i = 1}^{\poseig}}{(\partition{l})_{l = 1}^{T}}.
\end{equation}
This completes the proof.

\subsection{Proof of Lemma \ref{lemma:primal-dual-2}.}
First we shall conclude that the optimal sequence \( ( \eigval{i}\opt )_{i = 1}^{\poseig} \) is non-increasing. Let \( (\eigval{i}')_{i = 1}^{\poseig} \) to be the unique optimal solution to \eqref{eq: primalgeneral}, and let \( (\eigval{i}\opt)_{i = 1}^{\poseig} \) be the sequence obtained by rearranging the terms in non-increasing order. Since, \( \sum_{i = 1}^j \eigval{i}' \leq \sum_{i = 1}^j \eigval{i}\opt \) for every \( j = \indexes{\poseig - 1} \), and \( \sum_{i = 1}^{\poseig} \eigval{i}' = \sum_{i = 1}^{\poseig} \eigval{i}\opt \), it is immediate that the sequence \( (\eigval{i}\opt)_{i = 1}^{\poseig} \) is feasible for \eqref{eq: primalgeneral}. From the rearrangement inequality, we have:
\[
\sum_{i = 1}^{\poseig} \frac{s_i}{\eigval{i}\opt} \; \leq \; \sum_{i = 1}^{\poseig} \frac{s_i}{\eigval{i}'}.
\]
Therefore, the sequence \( (\eigval{i}\opt)_{i = 1}^{\poseig} \) is also optimal for \eqref{eq: primalgeneral}. This contradicts the fact that the optimization problem \eqref{eq: primalgeneral} admits a unique optimal solution, and the claim follows.

We shall prove the second assertion of the Lemma by contradiction. Without loss of generality, one can assume that \( j > i \), and suppose that \( s_j < s_i \) (because \( (s_i) \) is a non-increasing sequence). Let us define \( \epsilon \), \( (\eigval{l}')_{l = 1}^{\poseig} \) by 
\begin{equation*}
\begin{aligned}
\epsilon \Let \frac{\eigval{} (s_i - s_j) }{2 (s_i + s_j) }, \quad \text{and} \quad 
\eigval{l}' \Let 
\begin{cases}
\eigval{l}\opt   \quad &\text{for all \( l \neq i,j \)}, \\
\eigval{i}\opt + \epsilon &l = i,\\
\eigval{j}\opt - \epsilon   &l = j.
\end{cases}
\end{aligned}
\end{equation*}
Clearly, \( \epsilon > 0 \). It is also easy to verify that \( \eigval{l}' > 0 \) for all \( l = \indexes{\poseig} \), and \( \sum\limits_{l = 1}^k \eigval{l}\opt \leq \sum\limits_{l = 1}^k \eigval{l}' \) for all \( k = \indexes{\poseig} \). Therefore, the sequence\( ( \eigval{l}' )_{l = 1}^{\poseig} \) is feasible for the optimization problem \eqref{eq: primalgeneral}. We see that,
\[
\sum\limits_{l = 1}^{\poseig} \frac{s_l}{\eigval{l}'} - \sum_{l = 1}^{\poseig} \frac{s_l}{\eigval{l}\opt} \ = \  - \left( \frac{\epsilon^2 ( s_i + s_j ) }{\eigval{} (\eigval{} - \epsilon) (\eigval{} + \epsilon)} \right) \ < \ 0,
\]
which contradicts the fact that \( ( \eigval{l}\opt )_{l = 1}^{\poseig} \) is the optimal solution. The assertion of the lemma follows.

\subsection{Proof of Lemma \ref{lemma:primal-dual-3}.}
The first assertion of the lemma is straightforward, we shall prove the second assertion by contradiction. Let \( I \Let \{ \indexes{\poseig} \} \), and for every \( i = \indexes{\poseig} \) let us define, \( I_i \Let \big\{ j \in I \vert \eigval{j}\opt = \eigval{i}\opt \big\} \); since \( i \in I_i \) we have \( I_i \neq \emptyset \). Suppose that there exists \( i \in I \) such that for every \( j \in I_i \) we have \( s_j \neq s_{\pi\opt(i)} \). Exactly one of the following is true:
\begin{itemize}[leftmargin=*]
\item \textbf{case 1:} \( s_{\pi\opt(i)} > s_j \) for all \( j \in I_i \).\footnote{This is due to the fact that \( (s_i)_i \) is a non-increasing sequence.} \newline
In this case, \( \pi\opt(i) < j \) for all \( j \in I_i \), and therefore, there exists \( k \in I \) such that \( k < j \) for all \( j \in I_i \) and \( s_{\pi\opt(k)} < s_{\pi\opt(i)} \). Since \( k < j \) for all \( j \in I_i \), we have \( \eigval{k}\opt > \eigval{i}\opt \). We see that
\begin{equation*}
\sum_{l \neq i, k} \frac{s_{\pi\opt(l)}}{\eigval{l}\opt} + \frac{s_{\pi\opt(i)}}{\eigval{i}\opt} + \frac{s_{\pi\opt(k)}}{\eigval{k}\opt} > \sum_{l \neq i, k} \frac{s_{\pi\opt(l)}}{\eigval{l}\opt} + \frac{s_{\pi\opt(i)}}{\eigval{k}\opt} + \frac{s_{\pi\opt(k)}}{\eigval{i}\opt}.
\end{equation*}

\item \textbf{case 2:} \( s_{\pi\opt(i)} < s_j \) for all \( j \in I_i \). \newline
In this case, \( \pi\opt(i) > j \) for all \( j \in I_i \), and therefore, there exists \( k \in I \) such that \( k > j \) for all \( j \in I_i \) and \( s_{\pi\opt(k)} > s_{\pi\opt(i)} \). Since \( k > j \) for all \( j \in I_i \), we have \( \eigval{k}\opt < \eigval{i}\opt \). We see that
\begin{equation*}
\sum_{l \neq i, k} \frac{s_{\pi\opt(l)}}{\eigval{l}\opt} + \frac{s_{\pi\opt(i)}}{\eigval{i}\opt} + \frac{s_{\pi\opt(k)}}{\eigval{k}\opt} > \sum_{l \neq i, k} \frac{s_{\pi\opt(l)}}{\eigval{l}\opt} + \frac{s_{\pi\opt(i)}}{\eigval{k}\opt} + \frac{s_{\pi\opt(k)}}{\eigval{i}\opt}.
\end{equation*}
\end{itemize}
This contradicts the optimality of the permutation map \( \pi\opt \) since the new permuted sequence obtained by swapping \( s_{\pi\opt(i)} \) with \( s_{\pi\opt(k)} \) evaluates to a lower value of the objective function in both the cases. The assertion follows.

To see that \( \pi\opt \) is the optimal solution to \eqref{eq:opt-perm-map}, suppose \eqref{eq:opt-perm-map-solution} does not holds. Then there exist \( i,j \in I \) such that \( i < j \) and \( s_{\pi\opt(i)} < s_{\pi\opt(j)} \). Since \( i < j \), we have \( \eigval{i}\opt \geq \eigval{j}\opt \). Suppose that \( \eigval{i}\opt > \eigval{j}\opt \). Then we see that
\begin{equation*}
\sum_{l \neq i, j} \frac{s_{\pi\opt(l)}}{\eigval{l}\opt} + \frac{s_{\pi\opt(i)}}{\eigval{i}\opt} + \frac{s_{\pi\opt(j)}}{\eigval{j}\opt} > \sum_{l \neq i, k} \frac{s_{\pi\opt(l)}}{\eigval{l}\opt} + \frac{s_{\pi\opt(i)}}{\eigval{j}\opt} + \frac{s_{\pi\opt(j)}}{\eigval{i}\opt},
\end{equation*}a
contradicting the optimality of \( \pi\opt \). Thus, we conclude that \( \eigval{i}\opt = \eigval{j}\opt \). Therefore, we have \( I_i = I_j \), and from the previous assertion of the lemma we conclude that there exist \( i',j' \in I_i \) such that \( s_{\pi\opt(i)} = s_{i'} \) and \( s_{\pi\opt(j)} = s_{j'} \). Since \( s_{\pi\opt(i)} < s_{\pi\opt(j)} \), we have \( s_{i'} < s_{j'} \), but since \( \eigval{i'}\opt = \eigval{j'}\opt \), we conclude from Lemma \ref{lemma:primal-dual-2} that \( s_{i'} = s_{j'} \), which is a contradiction. Finally, since the sequence \( (s_l)_{l = 1}^{\poseig} \) is itself non-increasing, we immediately get \( s_{\pi\opt(l)} = s_l \) for all \( l = \indexes{\poseig} \), thereby completing the proof.

\subsection*{Acknowledgements}
We sincerely thank Prof.\ N.\ Khaneja for his valuable suggestions on the theory of majorization, and Prof.\ V.\ S.\ Borkar for illuminating discussions and encouragement.

\vskip 0.2in
\bibliographystyle{alpha}
\bibliography{ref}

\end{document}